\documentclass[b5paper,10pt]{article}
\usepackage{amsmath, amsthm, amssymb, fullpage, graphicx, centernot, url, enumerate, color, accents, bookmark}
\usepackage{hyperref, array}

\newtheorem{theorem}{Theorem}
\newtheorem{proposition}[theorem]{Proposition}
\newtheorem{lemma}[theorem]{Lemma}
\newtheorem{corollary}[theorem]{Corollary}
\newtheorem*{corollary*}{Corollary}
\theoremstyle{remark}
\newtheorem*{remark}{Remark}

\newcommand{\latin}[1]{#1} 

\newcommand{\C}{\mathbb{C}}
\newcommand{\R}{\mathbb{R}}
\newcommand{\N}{\mathbb{N}}
\newcommand{\Z}{\mathbb{Z}}
\newcommand{\pb}{\overline{p}} 
\newcommand{\Nt}{\widetilde{N}} 
\newcommand{\sigmat}{\widetilde{\sigma}}
\newcommand{\qh}{\hat{q}} 
\newcommand{\rh}{\hat{r}}
\newcommand{\sh}{\hat{s}}
\newcommand{\Ph}{\hat{\mathbb P}}

\newcommand{\equivalent}[1]{\overset{}{\underset{#1}{\sim}}}
\newcommand{\eps}{\varepsilon}

\newcommand*{\pc}[2][]{#1{p}_\mathrm{c}^{#2}}
\newcommand*{\future}[1][0]{(#1,\infty)}
\renewcommand*{\P}{\mathbb{P}}

\newcommand*{\E}{\mathbb{E}}

\newcommand*{\prob}[1]{\P(#1)}

\newcommand*{\onab}[3][1]{\{#3\}_{[x_{#1},x_{#2}]}}

\newcommand*{\onfuture}[2][0]{\{#2\}_{\future[#1]}}
\newcommand{\from}{\mathbin{\leftarrow}}
\newcommand{\ffrom}{\mathbin{\accentset{\!1}\leftarrow}} 
\newcommand{\tto}{\mathbin{\rightarrow}}
\newcommand{\fto}{\mathbin{\accentset{\!1}\rightarrow}} 
\newcommand{\collide}{\mathbin{\rightarrow\leftarrow}}
\newcommand{\triple}{\tto\!\bullet\!\from\bullet}
\newcommand{\nfrom}{\mathbin{\centernot\leftarrow}}
\newcommand{\nto}{\mathbin{\centernot\rightarrow}}

\newcommand{\go}{\accentset{\rightarrow}{\bullet}}
\newcommand{\come}{\accentset{\leftarrow}{\bullet}}
\newcommand{\stay}{\dot{\bullet}}
\newcommand{\Ngo}{\accentset{\rightarrow}{N}}
\newcommand{\Ncome}{\accentset{\leftarrow}{N}}
\newcommand{\Nstay}{\dot{N}}

\newcommand{\VARs}{r}

\newcommand{\VARr}{s}

\newcommand*{\bbr}[1]{\Bigl(#1\Bigr)}

\DeclareMathOperator{\rev}{rev}
\newcommand{\Fr}{\mathcal F}
\newcommand{\indic}{{\bf 1}}

\newcommand{\defeq}{\mathrel{\mathop:}=}

\newcommand{\st}{\,:\,} 
\newcommand{\s}{\mid} 
\DeclareMathOperator*{\limup}{\lim\!\!\uparrow}

\title{Three-speed ballistic annihilation: Phase transition and universality}

\author{John Haslegrave\\Mathematics Institute, University of Warwick, Coventry, UK
\and
Vladas Sidoravicius\\Courant Institute of Mathematical Sciences, New York\\
NYU-ECNU Institute of Mathematical Sciences at NYU Shanghai
\and
Laurent Tournier\\LAGA, Universit\'e Sorbonne Paris Nord,\\ CNRS, UMR 7539, 93430 Villetaneuse, France.
}

\begin{document}

\maketitle

\begin{abstract}
We consider ballistic annihilation, a model for chemical reactions first introduced in the 1980's physics literature. In this particle system, initial locations are given by a renewal process on the line, motions are ballistic --- i.e.\ each particle is assigned a constant velocity, chosen independently and with identical distribution --- and collisions between pairs of particles result in mutual annihilation. 

We focus on the case when the velocities are symmetrically distributed among three values, i.e.\ particles either remain static (with given probability~$p$) or move at constant velocity uniformly chosen among $\pm1$. We establish that this model goes through a phase transition at $p_c=1/4$ between a subcritical regime where every particle eventually annihilates, and a supercritical regime where a positive density of static particles is never hit, confirming 1990s predictions of Droz et al.~\cite{droz1995ballistic} for the particular case of a Poisson process. Our result encompasses cases where triple collisions can happen; these are resolved by annihilation of one static and one randomly chosen moving particle. 

Our arguments, of combinatorial nature, show that, although the model is not completely solvable, certain large scale features can be explicitly computed, and are universal, i.e. insensitive to the distribution of the initial point process. In particular, in the critical and subcritical regimes, the asymptotics of the time decay of the densities of each type of particle is universal (among exponentially integrable interdistance distributions) and, in the supercritical regime, the distribution of the ``skyline'' process, i.e.\ the process restricted to the \emph{last} particles to ever visit a location, has a universal description.  

We also prove that the alternative model introduced in~\cite{junge2}, where triple collisions resolve by mutual annihilation of the three particles involved, does not share the same universality as our model, and find numerical bounds on its critical probability. 


\end{abstract}

\maketitle

\begin{flushright}
\textit{To the dear memory of Vladas Sidoravicius, \\
who untimely passed away during the final preparation of this paper.}
\end{flushright}

\section{Introduction}

Originating in an effort to understand the kinetics of chemical reactions, several models of annihilating particle systems were introduced in the 1980's and 1990's in statistical physics. While most of the interest focused on diffusive motions, i.e.\ driven by random walks or Brownian motions (see for instance the celebrated results of Bramson and Lebowitz~\cite{bramson1991asymptotic} regarding two-type annihilation $A+B\to\emptyset$ on $\Z^d$, or by Arratia~\cite{arratia1983site} on one-type annihilation $A+A\to\emptyset$ on $\Z^d$), it was also observed, first by Elskens and Frisch~\cite{EF85} in a particular case, and later more systematically by Ben-Naim, Redner and Leyvraz~\cite{ben1993decay}, that the case of \textit{ballistic} motions (i.e.\ with constant velocity and direction) displayed very different behaviors and was particularly challenging to analyze. 
In this so-called ballistic annihilation process, particles start from the points of a homogeneous Poisson point process on the real line, and move at constant velocities that are initially chosen at random, independently and according to the same distribution; when two particles collide, they annihilate each other immediately. 

The distribution of velocities obviously plays a key role. The case when velocities take only two values, for instance $\pm1$, has attracted substantial interest, as it is not only physically relevant (cf.~\cite{EF85,krug1988universality}) but also combinatorially very tractable due to a reduction to random walks, and already displays interesting phenomena. For instance, although the question of survival of a given particle is extremely simple in this case, the global behavior of the cloud of surviving particles at large times is nontrivial: remaining particles of $+1$ and $-1$ velocities tend to form homogeneous aggregates whose asymptotic distribution can be computed. See, for instance~\cite{belitsky1995ballistic},~\cite{ermakov1998some}, until the recent~\cite{kovchegov2017dynamical}. 

At the other end of the spectrum, little is known on the case of continuous velocities. In physics literature, several mean-field analyses or computer simulations have been conducted to understand the decay of the concentration of particles~\cite{ben1993decay,trizac2002kinetics}. However, very few results are known rigorously. Some general observations were given in~\cite{sidoravicius-tournier}, especially on the symmetric case. A very intriguing combinatorial feature was also proved by Broutin and Marckert~\cite{BM19} on finite systems, namely that the law of the number of particles that survive forever, in the system restricted to $n$ particles starting at consecutive locations of a renewal process, does \textit{not} depend on the distribution of either velocities or interdistances. As explained in~\cite{BM19}, this property cannot be understood by a simple symmetry argument; accordingly, the proof of this inconspicuous property is surprisingly intricate, which suggests that this model has combinatorial interest beyond physics applications or sheer curiosity. In the mathematical community, interest in the problem was recently revived by Kleber and Wilson's popularisation of a puzzle~\cite{ibm} about a closely related ``bullet problem'', in which particles with independent uniformly distributed random speeds in $[0,1]$ leave the origin at integer times and are annihilated by collisions. In this setup, it is conjectured that there is a critical speed $s_{\mathrm{c}}\in(0,1)$ such that the first bullet survives with positive probability if and only if it flies faster than $s_c$. We refer to~\cite{junge} for more details and a partial answer in the case of discrete speed distributions. As explained in~\cite{junge,sidoravicius-tournier}, interchanging time and space in the bullet problem yields a one-sided instance of ballistic annihilation with a different distribution of speeds. 

The scope of the present paper lies within the intermediary case. More specifically, we show that arguably the simplest case beyond the case of two velocities, namely the case when velocities have a symmetric distribution on the set $\{-1,0,+1\}$, already goes through a phase transition that (contrary to the two-velocities case) is neither explained by a trivial symmetry nor a monotonicity. Let $p\in[0,1]$ denote the probability of a null velocity, in other words we assume that each particle independently is either static (with probability $p$), or moves at unit velocity either left or right (each with probability $\frac{1-p}2$). Krapivsky, Redner and Leyvraz~\cite{krapivsky1995ballistic}, who first considered this case in 1995, postulated the existence of a critical probability $\pc{}$, such that for $p<\pc{}$ every particle is eventually annihilated, whereas for $p>\pc{}$ a positive density of static particles survive forever. Based on simulations and a heuristic derived from considering the rate at which different types of collisions might be expected to occur, they conjectured that $\pc{}=1/4$. This conjecture was simultaneously strongly supported by intricate exact computations of Droz, Rey, Frachebourg and Piasecki \cite{droz1995ballistic} resolving related differential equations and also providing precise asymptotics for the decay of the densities of static and moving particles. However, these results are not entirely rigorous, and provide very little intuitive understanding of the process. Our first main result is a confirmation of the fact that $\pc{}=1/4$ and of an exact formula, first predicted by~\cite{droz1995ballistic}, for the asymptotic density of surviving static particles. As a consequence, this establishes that the survival probability of a given static particle is a monotonic and continuous function of $p$; neither of these properties was previously known. It is in particular important to underline that no monotonicity holds with respect to the introduction of more static particles into a configuration. The core of the proof is of combinatorial nature, relying on several symmetries in the model, and involves only simple computations, although a finer approximation argument is necessary to ensure some \latin{a priori} regularity and to properly conclude. 

While the previous works~\cite{krapivsky1995ballistic,droz1995ballistic} in the physics literature focused on the natural case of a Poisson process as initial distribution of locations, our proof remarkably holds irrespective of the distribution of the initial distances between particles, as long as they are i.i.d. Let us emphasize that these distances have a crucial role in the evolution of the system at ``microscopic scale''. In particular, general interdistances can produce occurrences of triple collisions (one static and two moving particles); in this case, which (to our knowledge) hadn't appeared in the physics models, we decide the outcome at random: with equal probability, both the static and one of the two moving particles annihilate, while the other moving particle survives (see Figure~\ref{fig:configurations}). With this definition, the system thus shows universal behavior at ``macroscopic scale'', i.e.\ belongs to the same phase for any distribution of interdistance. More deeply, our second main result (Theorem~\ref{thm:universality}) outlines a stronger, hitherto unsuspected, universality property that is reminiscent of the results of~\cite{BM19}. This property furthermore enables us to derive explicitly (Theorem~\ref{thm:asympt_density}) the asymptotic decay of densities of static and moving particles in the critical and subcritical regime, universally among i.i.d.\ and exponentially integrable interdistances. In the supercritical regimes, we show that densities converge exponentially fast to their limit, however the exact order depends on the distribution of interdistances and therefore does not follow from our methods, except for the particular case of constant interdistances. We also identify a fully universal object in the supercritical regime, here called the \emph{skyline process}: it consists intuitively in the distribution of the ``top shapes'' in the space-time representation (cf.~Figure~\ref{fig:skyline}) or, in terms of the process, in the set of indices and speeds of the last particles to ever visit a part of the environment, see Proposition~\ref{pro:skyline}. 

Some rigorous results were already known about this three-speed ballistic annihilation model. 
First, a simple argument shows that, for sufficiently large values of $p$, static particles have positive chance to survive: if a static particle is to be annihilated, then there must be an interval containing that particle which contains at least as many moving particles as static particles at time $0$, but for $p>1/2$ there is a positive probability that no such interval exists. Sidoravicius and Tournier~\cite{sidoravicius-tournier} proved that static particles survive with positive probability even when $p\ge1/3$; an alternative argument was also provided by Dygert et al.~\cite{junge}. Furthermore, a technique to numerically improve this bound was proposed by Burdinski, Gupta and Junge~\cite{junge2}. Despite these results, there have been no corresponding lower bounds, and proving that almost sure annihilation occurs at any small $p$ was therefore the central open question. Our results not only answer this question, but show that almost sure annihilation occurs if, and only if, $p\le1/4$.

We may also mention that our techniques adapt to the alternative version of the model recently introduced by Burdinski, Gupta and Junge \cite{junge2}, in which interdistances are constant and where, at triple collisions, all three particles are annihilated (cf.\ Figure~\ref{fig:configurations}, bottom). We in particular prove survival when $p\ge0.2406$ (improving over $p\ge 0.2870$ in~\cite{junge2}), implying that the phase transition of this variant happens strictly below $1/4$, confirming a conjecture of~\cite{junge2}. Furthermore, we also provide the first upper bound for the annihilation regime, although our techniques don't suffice in this case to establish the existence of a critical probability. One might intuitively suspect that the critical probability, if it exists, for this model would be strictly smaller than that for our model with constant interdistances. The rationale is that if the two processes are coupled then the latter is equivalent to injecting an additional moving particle into the former whenever a triple collision occurs, and these extra moving particles should make it easier to destroy stationary particles. However, the lack of monotonicity in the process means that it would be difficult to make this intuition rigorous. Furthermore, it would not give any insight into its relationship with the normal model with continuous interdistances, were it not for universality properties of the latter. The lack of universality in this alternative version when extended to generic discrete distributions also makes it difficult to predict its critical probability for constant distances.

Let us finally note that, since the prepublication of this paper, the robustness of our approach, more specifically of Section~\ref{sec:mainproof}, was illustrated in several directions by Junge, Lyu and co-authors in~\cite{junge2} and~\cite{junge3}. Both papers still consider the case of three velocities. The first~\cite{junge2} shows that the techniques, although insuffisant to exactly locate the phase transition in the three-speed \emph{asymmetric} case, can be adapted to get nontrivial bounds on the threshold and regularity statements on the survival probability. As to the second~\cite{junge3}, it investigates a generalization of the symmetric case where collisions may give birth to new particles in a random, symmetric way (akin to our definition of the outcome of triple collisions); although becoming too computationally involved to deal with general parameters, the techniques are again amenable to adaptation, enabling to get exact formulas for transition threshold in many cases (demonstrating in particular that these extensions do \emph{not} change phase at $p=1/4$). Regarding the model under focus in the present paper, we investigated further its remarkable universality property in a recent work~\cite{HT}: without being able to provide a satisfying explanation for it, we still identify other interesting instances of it in finite systems where it can be proved by alternative means and imply certain unexpected independence properties of possible interest to later studies.

\section{Definitions, notations and results}

Let us first define the model. In contrast with the above introduction, we will primarily restrict to particles starting from $(0,\infty)$ since our main result is best stated in that context.

Let $p\in[0,1]$, and let $m$ be a probability measure on $(0,\infty)$. 

On a probability space $(\Omega,\mathcal F,\P)$, let $(x_n)_{n\ge1}$ be a renewal process on $(0,\infty)$ whose interdistances are distributed according to~$m$, i.e.\ $x_1,x_2-x_1,x_3-x_2,\ldots$ are independent $m$-distributed random variables. Let also $(v_n)_{n\ge1}$ be a sequence of independent random variables on $\{-1,0,+1\}$, with same distribution given by
\[\P(v_n=0)=p,\quad \text{ and }\quad\P(v_n=-1)=\P(v_n=+1)=\frac{1-p}2,\]
and independent of $(x_n)_{n\ge1}$. Finally, let $(s_n)_{n\ge1}$ be a sequence of independent Rademacher random variables, 
which are also independent of $(x_n)_{n\ge1}$ and $(v_n)_{n\ge1}$

We interpret $x_1,x_2,\ldots$ as initial locations of particles on the real line $\R$, $v_1,v_2,\ldots$ as their initial velocities, and $s_1,s_2,\ldots$ as their ``spins''.
For any $i\ge1$, the spin $s_i$ will only play part in the process if $v_i=0$, in which case it will be used to resolve a potential \emph{triple collision} at $x_i$. In particular, spins can be ignored by the reader whose interest is in continuous interdistances. 

In notations, the particles will conveniently be referred to as $\bullet_1,\bullet_2,\ldots$, and particles with velocity $0$ 
will sometimes be called \emph{static} particles. 

Given the initial configuration $(x_n,v_n)_{n\ge1}$, the evolution of the process of particles may be informally described as follows (see also Figure~\ref{fig:configurations}): at time $0$, each particle $\bullet_n$ (for $n\in\N=\{1,2,\ldots\}$) starts at $x_n$, and then moves at constant velocity $v_n$ until, if ever, it collides with another particle. Collisions resolve as follows: where exactly two particles collide, both are annihilated; where three particles, necessarily of different speeds, collide, two are annihilated, and either the right-moving or left-moving particle survives (i.e.\ continues its motion unperturbed), according to the spin of the static particle involved. Note that each spin affects the resolution of at most one triple collision. Annihilated particles are considered removed from the system and do not take part in any later collision. 

Finally, we shall occasionally refer to the \emph{full-line process}, i.e., the corresponding process in which particles are released from $\{-x'_n\st n\ge1\}\cup\{0\}\cup\{x_n\st n\ge1\}$, where $(x'_n)_{n\ge1}$ is an independent copy of $(x_n)_{n\ge1}$. Accordingly, velocities $(v_n)_{n\in\Z}$ and spins $(s_n)_{n\in\Z}$ are independent, and distributed as above. 
In this case, the probability will be denoted $\P_\R$. 

\subsection{Formal definition of annihilations}
Let us give a proper definition of the trajectories of particles, which amounts to defining the annihilation times. This will in particular provide a justification for the almost sure existence of the model. Let us mention that the vocabulary and notations introduced here will not be used elsewhere in the paper. 

Let us define the \emph{virtual trajectory} of $\bullet_i$ as its trajectory in absence of any other particle, i.e.\ $t\in[0,\infty)\mapsto x_i+v_i t$. For $i<j$, let us say that $\bullet_i$ and $\bullet_j$ \emph{virtually collide} if their virtual trajectories intersect (i.e.\ if there is $t\ge0$ such that $x_i+v_it=x_j+v_jt$, which is equivalent to $v_i>v_j$); in this case, collision happens at time $t_{ij}=\frac{v_j-v_i}{x_i-x_j}$. Let us set $t_{ij}=+\infty$ when $v_i\leq v_j$, and $t_{ji}=t_{ij}$. 
Let $i<j$ be such that $\bullet_i$ and $\bullet_j$ virtually collide.
We define $I_{ij}$ to be the random interval of $\R$ of all points from where a particle, with some velocity in $\{-1,0,+1\}$, could start and virtually hit either $\bullet_i$ or $\bullet_j$ at or before time $t_{ij}$. The interval $I_{ij}$ is bounded because velocities are bounded, and more specifically, $I_{ij}$ is the interval $[x_i,x_j]$ if $(v_i,v_j)=(+1,-1)$, the interval $[x_i,x_j+(x_j-x_i)]$ if $(v_i,v_j)=(+1,0)$, and the interval $[x_i-(x_j-x_i),x_j]$ if $(v_i,v_j)=(0,-1)$. 
Denote by $N_{ij}$ the number of pairs $(k,k')\ne(i,j)$ such that $k<k'$, $x_k,x_{k'}\in I_{ij}$ and $t_{kk'}\leq t_{ij}$.
Almost surely $N_{ij}<\infty$, because $\{x_n\st n\ge1\}$ has no accumulation point. 

For all positive integers $i<j$, the property that $\bullet_i$ and $\bullet_j$ mutually annihilate, denoted by $\bullet_i\sim \bullet_j$, is defined in the following recursive manner: $\bullet_i\sim \bullet_j$ if $\bullet_i$ and $\bullet_j$ virtually collide (i.e.\ $v_i>v_j$), and 
\begin{itemize} 
	\item either $N_{ij}=0$,
	\item or for every particle $\bullet_k$ that virtually collides with $\bullet_i$ or $\bullet_j$ at or before time~$t_{ij}$, there is a particle $\bullet_{k'}$ such that $t_{kk'}<\min(t_{ik},t_{jk})$, and $\bullet_k\sim \bullet_{k'}$,
	\item or there is a particle $\bullet_h$ that virtually collides with $\bullet_i$ and $\bullet_j$ at time~$t_{ij}$, such that either $v_h=s_i=+1$ (implying $(v_i,v_j)=(0,-1)$) or $v_h=s_j=-1$ (implying $(v_i,v_j)=(+1,0)$), and for every particle $\bullet_k$ that virtually collides with $\bullet_i$ or $\bullet_j$ strictly before time~$t_{ij}$, there is a particle $\bullet_{k'}$ such that $t_{kk'}<\min(t_{ik},t_{jk})$, and $\bullet_k\sim \bullet_k'$.
\end{itemize}
If $\bullet_k$ and $\bullet_{k'}$ are such that $t_{kk'}<\min(t_{ik},t_{jk})\le t_{ij}<\infty$, then $x_k, x_{k'}\in I_{ij}$, $I_{kk'}\subset I_{ij}$, and most importantly $N_{kk'}<N_{ij}$, which shows that the above defining procedure eventually terminates. 

\begin{figure}\label{fig:configurations}
\begin{center}
\includegraphics[scale=0.5, page=3, trim=0in 0in -0.92in 0in]{configurations.pdf}

\includegraphics[scale=0.5, page=2]{configurations.pdf}

\includegraphics[scale=0.5, page=1, trim=0in 0in -0.05in 0in]{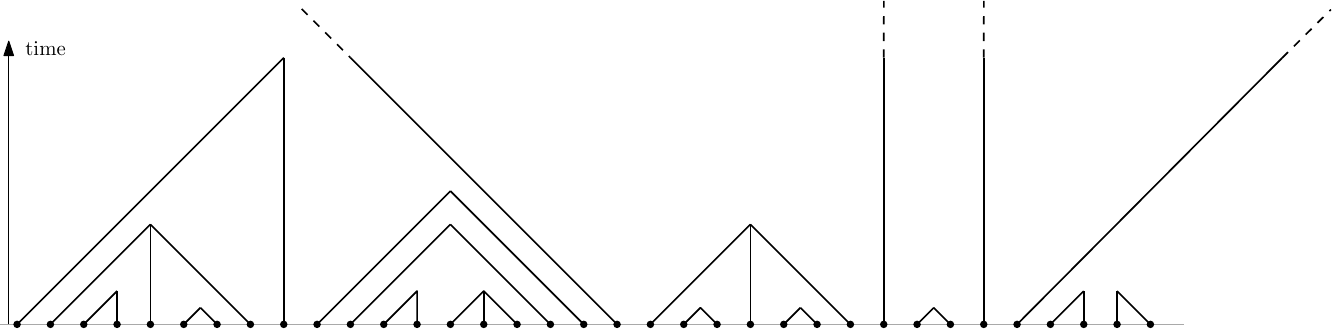}
\end{center}
\caption{Samples of the evolution of the system from a finite configuration, on a time-space diagram. The three figures share the same initial velocities. The top and middle figures correspond to our main model, with an atomless distribution $m$ of interdistances (top) or constant interdistances, i.e. $m=\delta_1$ (middle); the spins used to resolve triple collisions are indicated below the axis. The bottom figure represents, for constant interdistances, an alternative process where triple collisions resolve by mutual annihilation; this model is considered in Section~\ref{sec:discrete}.}
\end{figure}

\subsection{Notation}

Let us introduce convenient abbreviations to describe events related to the model. We use $\bullet_i$ (where $i\in\N$) for the $i$th particle, $\bullet$ for an arbitrary particle, and superscripts $\go$, $\stay$ and $\come$ to indicate that those particles have velocity $+1$, $0$ and $-1$ respectively. We write $\bullet_i\sim\bullet_j$ (for $i<j$ in $\N$) to indicate mutual annihilation between $\bullet_i$ and $\bullet_j$, however for readability reasons this notation will usually be replaced by a more precise series of notations: 
if $\bullet_i\sim\bullet_j$, we write
$\bullet_i\collide\bullet_j$, or redundantly $\go_i\collide\come_j$, when $v_i=+1$ and $v_j=-1$,  $\bullet_i\tto\bullet_j$ when $v_i=+1$ and $v_j=0$, and $\bullet_i\from\bullet_j$ symmetrically. Note that in all cases this notation excludes the case where $\bullet_i$ and $\bullet_j$ take part in a triple collision but one of them survives. Additionally, we write $x\from\bullet_i$ (for $i\in\N$ and $x\in\R$) to indicate that $\bullet_i$ crosses location~$x$ from the right (i.e.\ $v_i=-1$, $x<x_i$, and $\bullet_i$ is not annihilated when or before it reaches $x$), and $x\ffrom\bullet_i$ if $\bullet_i$ is \emph{first} to cross location~$x$ from the right. Symmetrically, we write $\bullet_i\tto x$ and $\bullet_i\fto x$. 

For any interval $I\subset(0,\infty)$, and any condition $C$ on particles, we denote by $(C)_I$ the same condition for the process restricted to the set $I$, i.e.\ where all particles outside $I$ are removed at time $0$ (however, the indices of remaining particles are unaffected by the restriction). For short, we write $\{C\}_I$ instead of $\{(C)_I\}$, denoting the event that the condition $(C)_I$ is realized. 


\subsection{Results}

Our main results apply to the model as defined above, i.e.\ where triple collisions are resolved by the annihilation of the static particle and of one randomly chosen moving particle, while remarks on the alternative discretized model of \cite{junge2} will be deferred to a later section.

\begin{theorem}\label{thm:main}
The model undergoes a phase transition at $\pc{}=1/4$. More precisely, the probability that $0$ is reached by a particle on $(0,\infty)$ is given, for all $p\in[0,1]$, by
\[q\defeq\P(0\from\bullet)=\begin{cases} 1 & \text{if $p\le1/4$}\\\frac1{\sqrt p}-1&\text{if $p>1/4$.}\end{cases}\]
\end{theorem}

This result has the following immediate interpretation in the full-line process: 
\begin{itemize}
	\item if $p\le1/4$, then a.s.\ all static  particles (i.e., with velocity $0$) are annihilated; 
	\item if $p>1/4$, then a.s.\ infinitely many static particles survive. More precisely, due to shift invariance, each static particle has same positive probability  to survive forever, which is given by \[\theta(p)=(1-q)^2=\Big(2-\frac1{\sqrt p}\Big)^2\] (the first equality follows by left-right symmetry and independence: if the particle at $0$ is static, its survival on the full line means that no particle crosses $0$ from either left or right), hence by ergodicity there is a density $p\theta(p)>0$, among particles, of surviving static particles. 
\end{itemize}

One can also see (cf.\ for instance~\cite{sidoravicius-tournier}) that, for $p\le1/4$, a.s.\ infinitely many particles cross every $x\in\R$ from both left and right, but only finitely many do so for $p>1/4$. 

\begin{figure}
\begin{tabular}{m{6cm}m{6cm}}
\includegraphics[width=5cm]{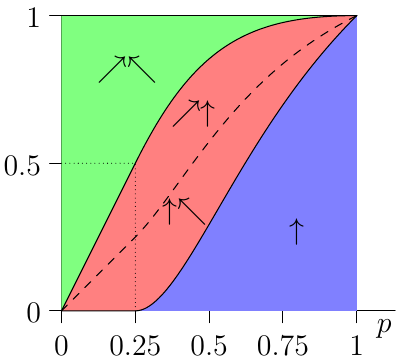}
&
\includegraphics[width=5cm]{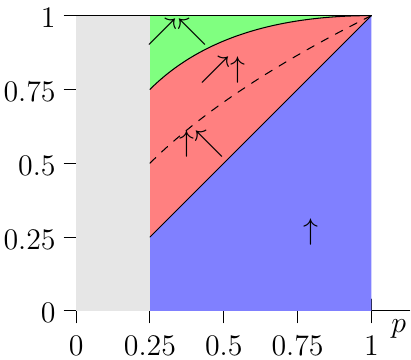}
\end{tabular}
\caption{(Left) Overall densities of each type of shape. (Right) Densities of each type of shape in the skyline process (see before Proposition~\ref{pro:skyline}).}\label{fig:shapes}
\end{figure}

While a Poisson point process is the most natural initial distribution for the particles, and was indeed the one considered in physics literature, the only property of the process $(x_n)_{n\ge1}$ of starting locations that we use is the fact that the intervals between particles are i.i.d. 
For this class of models, Theorem~\ref{thm:main} shows that $\theta(p)$ is universal. As a consequence, the relative frequencies of the possible ``shapes'', that is of surviving static particles, annihilations between static and moving particles, and annihilations between two moving particles, are also universal, and proportional to $p\theta(p)$, $p-p\theta(p)$ and $\frac12(1-2p+p\theta(p))$ respectively, as shown in Figure~\ref{fig:shapes} (Left).

Surprisingly, it turns out that a stronger form of universality holds, which is very reminiscent of the main result of~\cite{BM19}. Let $A$ be the random variable given by 
\[A\defeq\min\{k\ge1:0\from\bullet_k\}\] 
on the event $\{0\from\bullet\}$, and $A=\infty$ otherwise.

\begin{theorem}\label{thm:universality}
The distribution of $A$ does not depend on the distribution $m$ of interdistances. 
Furthermore, its generating series $f:x\mapsto f(x)=\E[x^A\indic_{\{A<\infty\}}]$ 
satisfies, for all $x\in[-1,1]$, 
\begin{equation}\label{eq:generating_series}
px f(x)^4-(1+2p)xf(x)^2+2f(x)-(1-p)x=0.
\end{equation}
\end{theorem}

Let us emphasize that the law of the pairing of particles by annihilation \textit{does} depend on $m$, which makes the above result remarkable. 
As was the case for the monotonicity of $q$, this universality follows a posteriori from explicit computation, while a more direct understanding is still missing. 

The above implicit equation for $f(x)$ in particular enables us to compute the asymptotic decay of densities of particles. Denoting by $c_0(t)$ (resp.\ $c_+(t)$) the density of static (resp.\ speed $+1$) particles at time $t$ on the full-line (see Section~\ref{sec:asymptotics} for details), we have in particular the following result. 

\begin{theorem}[Asymptotics of the density of particles]\label{thm:asympt_density}Assume the law $m$ of distance between particles to be exponentially integrable (i.e.\ $\int e^{\eta x}dm(x)<\infty$ for some $\eta>0$) and have mean $1$. Then, for some $c=c(p)>0$, as $t\to\infty$, 
\[c_0(t)=\begin{cases}
\bigl(\frac{2p}{\pi(1-4p)}+o(1)\bigr)\,t^{-1}	& \text{if $p<1/4$,}\\
\bigl(\frac{2^{2/3}}{4\Gamma(2/3)^2}+o(1)\bigr)\,t^{-2/3}	& \text{if $p=1/4$,}\\
(2\sqrt p-1)^2-o(e^{-ct})	&\text{if $p>1/4$,}
\end{cases}\]
and
\[c_+(t)=\begin{cases}
\bigl(\frac1{\sqrt\pi}\sqrt{1-4p}+o(1)\bigr)\,t^{-1/2}	&\text{if $p<1/4$,}\\
\bigl(\frac{2^{2/3}}{8\Gamma(2/3)^2}+\frac{3 }{8\Gamma(1/3)}+o(1)\bigr)\,t^{-2/3}	&\text{if $p=1/4$,}\\
o(e^{-ct})	&\text{if $p>1/4$.}
\end{cases}\]
Furthermore, when $m=\delta_1$, if $p>1/4$, 
\begin{equation}\label{eq:equiv_c0_discrete}
(2\sqrt p-1)^2-c_0(n)\equivalent{\substack{n\to\infty\\n\text{ even}}}2c_+(n) \equivalent{\substack{n\to\infty\\n\text{ even}}}\frac{9p}{2\sqrt p+1}\sqrt{\frac{8(1-p)(8p+1)}{\pi(4p-1)^5}}R^{-(n+1)}n^{-3/2}
\end{equation}
where $R=R(p)=\frac3{8p+1}\sqrt{\frac {3p}{1-p}}$.
\end{theorem}

Let us note that the assumption of exponential integrability of $m$ is not purely technical, since it ensures the above \emph{universal} asymptotic behavior. As explained in the remark p.~\pageref{rmk:exponential_integrability}, a mere integrability assumption could lead to different asymptotics. Let us also mention that, on explicit examples, a (computationaly demanding) way to obtain asymptotics for densities of particles would rely on studying Laplace transform using Tauber theory, starting from the implicit Equation~(1) from~\cite{HT}; this was, informally, the approach used in~\cite{droz1995ballistic} to correctly predict the above asymptotics, in the exponential case. 


\begin{figure}
\includegraphics[width=125mm, page=2]{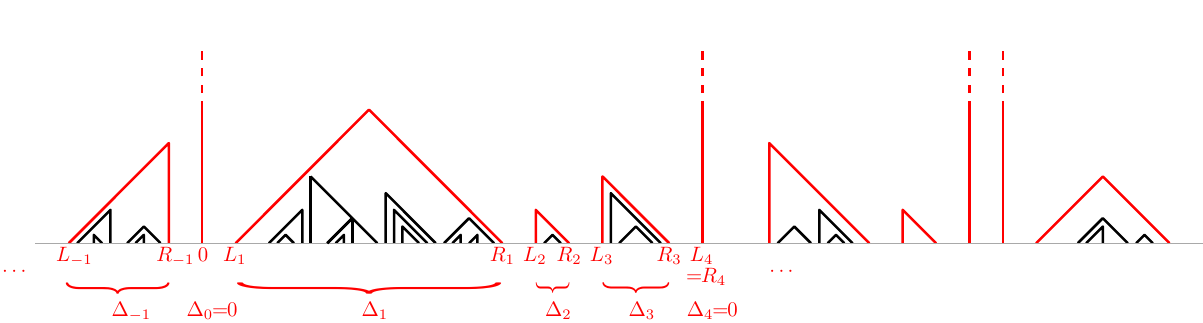}
\caption{Graphical space-time interpretation of the skyline process (see before Proposition~\ref{pro:skyline}), here highlighted in red. 
Note that the system is assumed to be defined on the whole line, although only a piece is represented, and the four static particles not yet annihilated here are assumed not to be hit by any other particle at any further time. }\label{fig:skyline}
\end{figure}

Our last result regards the process on the full-line.
Assume $q<1$, i.e.\ $p>1/4$.  In this case, every location on the line is visited a finite number of times. We are interested in the description of the particles that are the \emph{last} visitors of a point. 

Conditional on $\bullet_0$ being a never-annihilated static particle, i.e.\ under the probability that we will denote by  $\P_\R(\cdot\mid\bullet\nto\,\stay_0\nfrom\bullet)$, define the sequences $(L_n)_{n\ge0}$ and $(R_n)_{n\ge0}$ in $\Z$, and $(\Sigma_n)_{n\ge0}$ in the 4-element set $\{\uparrow,\nearrow\uparrow,\uparrow\nwarrow,\nearrow\nwarrow\}$, by $L_0=R_0=0$, $\Sigma_0=\uparrow$ and, for all $k\ge0$, $L_{k+1}=R_k+1$, and
\begin{itemize}
	\item if $v_{L_{k+1}}=0$ and $\stay_{L_{k+1}}\nfrom\come$, then $R_{k+1}=L_{k+1}$ and $\Sigma_{k+1}=\uparrow$;
	\item else, let $R_{k+1}$ be the index of the particle that annihilates with $\bullet_{L_{k+1}}$, i.e.\ $\bullet_{L_{k+1}}\sim\bullet_{R_{k+1}}$, and $\Sigma_{k+1}=\nearrow\uparrow$ (resp.\ $=\nearrow\nwarrow$, resp.\ $=\uparrow\nwarrow$) if $(v_{L_{k+1}},v_{R_{k+1}})=(+1,0)$ (resp.\ $=(+1,-1)$, resp.\ $=(0,-1)$). 
\end{itemize}
Note that the condition $\{\stay_0\nfrom\bullet\}$ implies that $v_{L_{k+1}}$ is not equal to $-1$, since otherwise the above construction would lead to $0\from\come_{L_{k+1}}$.  
We define $(L_k, R_k, \Sigma_k)_{k<0}$ symmetrically. 

In reference to space-time representation, see Figure~\ref{fig:skyline}, we call the sequence $(R_k-L_k,\Sigma_k)_{k\in\Z}$ the \emph{skyline} process. Note that the sequences $(L_k)_k$ and $(R_k)_k$ can be recovered from $(R_k-L_k)_k$, hence the skyline process indeed contains the information on indices and velocities of the last particles to ever visit some location. 

\begin{proposition}\label{pro:skyline}\begin{enumerate}[a)]
\item\label{p:s:a} Under $\P_\R(\cdot\mid\bullet\nto\,\stay_0\nfrom\bullet)$, $((R_k-L_k, \Sigma_k))_{k\in\Z\setminus\{0\}}$ are  i.i.d.\ random variables, whose distribution does not depend on $m$. More precisely, each is distributed as $(\Delta,\Sigma)$ where the law of $\Sigma$ is given by
\begin{align*}
& P(\Sigma=\uparrow)=p,\qquad P(\Sigma=\nearrow\uparrow)=P(\Sigma=\uparrow\nwarrow)=\sqrt p(1-\sqrt p),\\ 
& P(\Sigma=\nearrow\nwarrow)=(1-\sqrt p)^2
\end{align*}
and the law of $\Delta$, conditional on $\Sigma$ is given by, for each $n\ge0$, 
\begin{align*}
&P(\Delta=n\mid \Sigma=\uparrow) = \delta_{0}(n),\\
&P(\Delta=n\mid \Sigma=\uparrow\nwarrow)= P(\Delta=n\mid \Sigma=\nearrow\uparrow)
=\P(A=n\mid A<\infty),\\
&P(\Delta=n \mid \Sigma=\nearrow\nwarrow) =\P(\go_1\collide\come_{n+1}\mid \go_1\collide\come).
\end{align*}
\item\label{p:s:b} Under $\P_\R(\cdot\mid\bullet\nto\,\stay_0\nfrom\bullet)$, the random variables $\lvert v_{L_k}\rvert$  and $\lvert v_{R_k}\rvert$, ${k\in\Z\setminus\{0\}}$, are i.i.d.\ Bernoulli random variables with parameter $1-\sqrt{p}$.
\end{enumerate}
\end{proposition}

Remarks. 
\begin{itemize}
	\item Note that \ref{p:s:b}) offers a remarkably simple equivalent description of the distribution of $(\Sigma_k)_{k\in\Z\setminus\{0\}}$. Indeed, for all $k\ne0$ the only possible couples of velocities $(v_{L_k}, v_{R_k})$ are $(0,0)$, $(+1,0)$, $(0,-1)$ and $(+1,-1)$ (corresponding to $\Sigma_k=\uparrow$, $\nearrow\uparrow$, $\uparrow\nwarrow$ and $\nearrow\nwarrow$ respectively), which are characterized by their absolute values. 
	\item It will follow from the proof of Theorem~\ref{thm:universality} that the probabilities $\P(0\ffrom\come_n)$ and $\P(\go_1\collide\come_{n+1})$ can be computed by induction, hence the same holds for the distribution of $\Delta$ since both $\P(A<\infty)=q=1/{\sqrt p}-1$ and $\P(\go_1\collide\come)=(1-\sqrt p)^2$ are explicit (see the end of the proof of Proposition~\ref{pro:skyline}). 
	\item The distribution of $\Sigma$ is illustrated in Figure~\ref{fig:shapes} (Right). It is in particular interesting to notice that $P(\Sigma=\uparrow)$ converges to $\frac14>0$ as $p\searrow1/4$, while the density of surviving static particles converges to $0$ in the same time by continuity of $q$. In other words, even though surviving static particles are scarce in barely supercritical systems, they still represent a positive ($>1/4$) proportion of the shapes in the skyline. This contrast is explained by an increase in the expected size of shapes: $E[\Delta\s \Sigma=\nearrow\uparrow]=\E[A\s A<\infty]$ grows to $+\infty$ as $p\searrow1/4$ (cf.\ Theorem~\ref{thm:asymptotics}). 
\end{itemize}

\subsection{Structure of the proofs and organization of the paper}

The results presented in this paper provide in principle two approaches for proving the main theorem: one that appeals to both algebraic and topological arguments and another one that is more purely algebraic. Our main focus is on the first one, that is practically tractable, and also more robust. It indeed enables us to get some information about the alternative discrete model of \cite{junge2} (cf.~Section~\ref{sec:discrete}) and, after the prepublication of the present paper, it was also adapted by Junge and Lyu~\cite{junge-lyu-asymmetric} to study the asymmetric case, and later by Benitez, Junge et al.~\cite{junge3} to study variants of the symmetric case which extend the randomized resolution of triple collisions of our model to all types of collisions. 

The proof of Theorem~\ref{thm:main} decomposes into two parts, gathered in Section~\ref{sec:mainproof}. First, and most importantly, in Subsection~\ref{sec:algebraic}, using symmetries and independence to decompose the event $\{0\from\bullet\}$ in a ``recursive'' way, we are able to prove that, for any $p\in(0,1)$, $q$ solves an algebraic equation. This equation has two roots, namely $1$ and $1/{\sqrt p}-1$. Although this entails $q=1$ when $p\le1/4$, this doesn't prove the converse. 
For the latter, \latin{a priori} regularity properties of $q$ (or $\theta$) as a function of $p$ are needed, which are the subject of Subsection~\ref{sec:topological}. Unfortunately, it is not possible to rely on monotonicity since the apparent lack thereof is precisely a major difficulty in this model. This is instead achieved using finitary conditions characterizing the survival phase, together with the previous dichotomy. 
While proving regularity (specifically, lower or upper semicontinuity) using finitary conditions is classical in statistical physics (cf.\ for instance~\cite[Section 8.3]{zbMATH00195103}), we don't know of another situation where this enables to locate the transition threshold. 

Although the algebraic equation solved by $q$ is actually a particular case of Equation~\eqref{eq:generating_series} (indeed, $q=f(1)$), we keep its proof separate from the proof of Theorem~\ref{thm:universality} for the reason of greater robustness explained above. 

Interestingly, the above Equation~\eqref{eq:generating_series} (that is established by similar symmetry and independence arguments as in Subsection~\ref{sec:algebraic}) can be seen a posteriori to contain sufficient information to imply Theorem~\ref{thm:main} by itself, thereby circumventing the needs to ensure \latin{a priori} regularity of $q$ (in $p$) by exploiting instead the immediate (and considerable) regularity of $f(x)$ in $x$, so as to identify the probabilistically  meaningful root. See the remark at the end of Section~\ref{sec:proof_universality} for a more precise explanation of this algebraic viewpoint, which rather remains a theoretical approach than an actual alternative proof. 

Section~\ref{sec:proof_universality} proves Theorem~\ref{thm:universality} about the universality of the law of $A$. Section~\ref{sec:skyline} proves Proposition~\ref{pro:skyline} about the skyline. Finally, Section~\ref{sec:asymptotics} deduces Theorem~\ref{thm:asympt_density} from Theorem~\ref{thm:universality}. 


Finally, Section~\ref{sec:discrete} states and discusses the analogous, yet weaker, results about the alternative discretized version of the model (i.e.\ under $\P_\N$) that follow from adapting the previous arguments. 

\section{Proof of Theorem~\ref{thm:main}: Phase transition}\label{sec:mainproof}

\subsection{Algebraic identities}\label{sec:algebraic}

In this subsection, we prove
\begin{proposition}\label{prop:dichotomy}
For all $p\in(0,1)$, either $q=1$ or $q=\frac1{\sqrt p}-1$. In particular, $q=1$ if $p\le1/4$. 
\end{proposition}

Let us introduce temporary notation for the probabilities involved in the next two lemmas:
\begin{gather*}\VARs\defeq\prob{(0\nfrom\bullet)\wedge(\go_{1}\tto\stay)},\\
\VARr\defeq\prob{(0\from\bullet)\wedge(\go_{1}\tto\stay)}.\end{gather*}
These probabilities can be expressed in terms of $p$ and $q$ as follows:
\begin{lemma}\label{flip}
\begin{enumerate}
[a)]
	\item\label{s-relation} $\VARs=pq(1-q)$.
	\item\label{r-relation} $\VARr =\frac12pq^2$. 
\end{enumerate}
\end{lemma}

\begin{proof} 
For any integer $k\ge2$, and any configuration $\omega=((x_n,v_n,s_n))_{n\ge1}$, we define $\rev_k(\omega)$ to be the configuration obtained by reversing the interval $[x_1,x_k]$, that is, a particle at position $x\in[x_1,x_k]$ in $\omega$ corresponds to a particle at position $x_1+(x_k-x)$ moving in the opposite direction and having the opposite spin in $\rev_k(\omega)$, while particles outside $[x_1,x_k]$ are the same in $\omega$ and $\rev_k(\omega)$. Note that the speeds $+1$ and $-1$ have same probability $\frac{1-p}2$, the spins $+1$ and $-1$ independently have the same probability $1/2$, and given $\omega$ on $\R\setminus(x_1,x_k)$, the law of the configuration in $(x_1,x_k)$ of both $\omega$ and $\rev_k(\omega)$ is given by $k-2$ particles separated by $k-1$ independent distances with law $m$ and conditioned on having a total sum equal to $x_k-x_1$. Consequently $\rev_k$ is measure-preserving for any distribution $m$.

\ref{s-relation}) Let $k\ge2$. Note that $\{0\nfrom\bullet\}\cap\{\go_{1}\tto\stay_k\}=\{\go_{1}\tto\stay_k\}_{[x_1,x_k]}\cap\{x_k\nfrom\bullet\}_{(x_k,\infty)}$, since on the event
$\{\go_{1}\tto\stay_k\}_{[x_1,x_k]}\cap\{x_k\from\bullet_j\}_{(x_k,\infty)}$, $\bullet_j$ will either annihilate $\bullet_k$ or reach $0$.

The map $\rev_k$ induces a bijection between $\{\go_{1}\tto\stay_k\}_{[x_1,x_k]}\cap\{x_k\nfrom\bullet\}_{(x_k,\infty)}$ and 
$\{\stay_1\from\come_k\}_{[x_1,x_k]}\cap\{x_k\nfrom\bullet\}_{(x_k,\infty)}$. Since $\rev_k$ preserves the measure, it follows that 
\begin{align}
\P(\{0\nfrom\bullet\}\cap\{\go_{1}\tto\stay_k\}) 
	& =\P\bigl(\{\stay_1\from\come_k\}\cap\{x_k\nfrom\bullet\}_{(x_k,\infty)}\bigr)\nonumber\\
	& =\P(\stay_1\from\come_k)\P\bigl((x_k\nfrom\bullet)_{(x_k,\infty)}\bigr) = \P(\stay_1\from\come_k)(1-q),\label{eq:lemma_a}
\end{align}
where the second equality comes from the fact that the two events $\{\stay_1\from\come_k\}$ and $\{x_k\nfrom\bullet\}_{(x_k,\infty)}$ depend on the configuration on the disjoint intervals $[x_1,x_k]$ and $(x_k,\infty)$ respectively, and the last equality comes from translation invariance. Summing over $k\ge2$ finally gives, due to~\eqref{eq:lemma_a},
\[\P(\{0\nfrom\bullet\}\cap\{\go_{1}\tto\stay\})= \P(\stay_1\from\bullet)(1-q) = pq(1-q).\]

\ref{r-relation}) Let $n\ge k\ge2$. 
The event $\{0\ffrom\bullet_n\}\cap\{\go_{1}\tto\stay_k\}$ happens if, and only if, the $n$th particle is the first to reach $x_k$ from the right and the leftmost particle reaches the static particle at $x_k$ either strictly before this, i.e.\ $x_k-x_1<x_n-x_k$, or at the same time and this triple collision resolves by annihilation of the leftmost particle, i.e.\ $x_k-x_1=x_n-x_k$ and $s_k=-1$. Thus, 
\begin{align*}
\lefteqn{\{0\ffrom\bullet_n\}\cap\{\go_{1}\tto\stay_k\}}\\
	& = \{(\stay_k)\wedge(\go_1\fto x_k)_{[x_1,x_k)}\wedge(x_k\ffrom\bullet_n)_{(x_k,x_n]}\wedge(x_k-x_1<x_n-x_k)\}\\
	& \cup \{(\stay_k)\wedge(\go_1\fto x_k)_{[x_1,x_k)}\wedge(x_k\ffrom\bullet_n)_{(x_k,x_n]}\wedge(x_k-x_1=x_n-x_k)\wedge(s_k=-1)\}
\end{align*}
and, applying $\rev_k$, we readily have
\begin{align*}
\lefteqn{\rev_k(\{0\ffrom\bullet_n\}\cap\{\go_{1}\tto\stay_k\})}\\
	& = \{(\stay_1)\wedge(x_1\ffrom\come_k)_{(x_1,x_k]}\wedge(x_k\ffrom\bullet_n)_{(x_k,x_n]}\wedge(x_k-x_1<x_n-x_k)\}\\
	& \cup \{(\stay_1)\wedge(x_1\ffrom\come_k)_{(x_1,x_k]}\wedge(x_k\ffrom\bullet_n)_{(x_k,x_n]}\wedge(x_k-x_1=x_n-x_k)\wedge(s_1=+1)\}.
\end{align*}
Because $\rev_k$ preserves the measure, it follows that
\begin{align*}
\lefteqn{\P\bigl((0\ffrom\bullet_n)\wedge(\go_{1}\tto\stay_k)\bigr)}\\
	& = \P\bigl((\stay_1)\wedge(x_1\ffrom\come_k)_{(x_1,x_k]}\wedge(x_k\ffrom\come_n)_{(x_k,x_n]}\wedge(x_k-x_1<x_n-x_k)\bigr)\\
	&\!\! + \P\bigl((\stay_1)\wedge(x_1\ffrom\come_k)_{(x_1,x_k]}\wedge(x_k\ffrom\come_n)_{(x_k,x_n]}\wedge(x_k-x_1=x_n-x_k)\wedge(s_1=+1)\bigr)\\
	& = p\P\bigl((x_1\ffrom\come_k)_{(x_1,x_k]}\wedge(x_k\ffrom\come_n)_{(x_k,x_n]}\wedge(x_k-x_1<x_n-x_k)\bigr)\\
	&\!\! + \frac12p\P\bigl((x_1\ffrom\come_k)_{(x_1,x_k]}\wedge(x_k\ffrom\come_n)_{(x_k,x_n]}\wedge(x_k-x_1=x_n-x_k)\bigr),
\end{align*}
where the last equality follows from the independence of $(v_1,s_1)$ from $x_1$ and the configuration on $(x_1,x_n]$. The environment on $(x_k,x_n]$ is independent of the environment on $(x_1,x_k]$ and distributed as the environment on $(0,x_{n-k+1}]$. Therefore, summing over all values of $n(\ge k)$ and then of $k\ge1$, 
\begin{align}
\P\bigl((0\from\bullet)\wedge(\go_{1}\tto\stay)\bigr)
	& = p\P\bigl((D< \infty)\wedge(D'<\infty)\wedge(D<D')\bigr)\nonumber\\
	& + \frac12p\P\bigl((D<\infty)\wedge(D'<\infty)\wedge(D=D'))\bigr)\label{half-half}
\end{align}
where $D,D'$ are independent random variables distributed as
\[D\defeq\min\{x_k\st 0\from\bullet_k\}\in(0,\infty).\]
Since by symmetry $\P(D<D'<\infty)=\frac12\P(D\ne D', D<\infty,D'<\infty)$, combining the previous probabilities yields
\begin{align*}
\P\bigl((0\from\bullet)\wedge(\go_{1}\tto\stay)\bigr)
	& = \frac12 p\P(D<\infty)^2=\frac12pq^2.\qedhere
\end{align*}
\end{proof}

\begin{lemma}\label{lem:identity}
$q=\frac{1-p}2+pq^2+\VARr+\Bigl(\frac{1-p}2-\VARs-\VARr\Bigr)q$.
\end{lemma}

\begin{proof}
Conditioning on the velocity of the leftmost particle, we have
\begin{equation}q=\frac{1-p}{2}\prob{0\from\bullet\mid\come_1}
+p\prob{0\from\bullet\mid\stay_1}
+\prob{(0\from\bullet)\wedge(\go_{1})}.\label{cond1}\end{equation}
Clearly if the leftmost particle moves left it will reach $0$, hence $\prob{0\from\bullet\mid\come_1}=1$. If the leftmost particle is static, it is annihilated with probability $q$, since this
equals $\prob{\onfuture[x_1]{x_1\from\bullet}}$. Note, however, that this event occurs if and only if $\onab[1]{j}{\stay_1\from\come_j}$ occurs 
for some $j$, since the progress of a left-moving particle cannot be affected by particles further to the right. 
Given that $\onab[1]{j}{\stay_1\from\come_j}$ occurs,
a particle reaches $0$ if and only if $\onfuture[x_j]{x_j\from\bullet}$ also occurs, since the fact that $\bullet_j$ is left-moving and annihilates
$\bullet_1$ means that no particle from the right of $\bullet_j$ can encounter any particles after reaching $x_j$. Clearly $\onfuture[x_j]{x_j\from\bullet}$
is independent of $\onab[1]{j}{\stay_1\from\come_j}$ and has probability $q$, so $\prob{0\from\bullet\mid\stay_1}=q^2$.

If the leftmost particle moves right, it must eventually be annihilated (see e.g.\ \cite[Lemma 3.3]{sidoravicius-tournier}). Thus we have
\begin{equation}\frac{1-p}{2}=\prob{\go_{1}\tto\stay}+\prob{\go_{1}\collide\come}.\label{hits}\end{equation}
Conditioning on how the leftmost particle is annihilated, we have
\begin{equation}
\prob{(0\from\bullet)\wedge(\go_{1})}=\prob{(0\from\bullet)\wedge(\go_{1}\tto\stay)}
+\prob{0\from\bullet\mid\go_{1}\collide\come}\prob{\go_{1}\collide\come}.\label{cond}
\end{equation}
Now $\prob{0\from\bullet\mid\go_{1}\collide\come}=q$, since, 
given that $\come_j$ annihilates $\go_{1}$,
$\{0\from\bullet\}$ if and only if $\onfuture[x_j]{x_j\from\bullet}$. Thus \eqref{cond} becomes
\begin{equation}
\prob{(0\from\bullet)\wedge(\go_{1})}=\prob{(0\from\bullet)\wedge(\go_{1}\tto\stay)}
+q\prob{\go_{1}\collide\come},\label{q-if-right}
\end{equation}
and, combining \eqref{hits} and \eqref{q-if-right},
\begin{align}
\prob{(0\from\bullet)\wedge(\go_{1})}&=s+q\bbr{\frac{1-p}{2}-\prob{\go_{1}\tto\stay}}\nonumber\\
&=s+q\bbr{\frac{1-p}{2}-r-s}.\label{use-r}\end{align}
Combining \eqref{use-r} and \eqref{cond1} gives
\begin{align*}q&=\frac{1-p}{2}+pq^2+s+q\bbr{\frac{1-p}{2}-r-s}.\qedhere\end{align*}
\end{proof}

\begin{proof}[{Proof of Proposition~\ref{prop:dichotomy}}]
Combining the previous two lemmas yields immediately the equation
\[0=1-q-p-pq+pq^2+pq^3\]
hence
\[0=1-q-p(1+q-q^2-q^3)
=(1-q)(1-p(1+q)^2),\]
implying, since $q\ge0$, that either $q=1$ or $q=\frac1{\sqrt p}-1$. Since $q\le1$, we conclude that $q=1$ when $p\le1/4$. 
\end{proof}

\subsection{A priori regularity properties}\label{sec:topological}

Let us prove the following result, which in combination with Proposition~\ref{prop:dichotomy} immediately gives Theorem~\ref{thm:main}. Remember that $\theta(p)=(1-q)^2$ is the survival probability of a static particle in the full line process. 

\begin{proposition}\label{prop:connectivity}
For all $p\in(\frac14,1)$, $\theta(p)>0$. 
\end{proposition}

The proof follows from the two lemmas below. These lemmas respectively rely on two different characterizations of the supercritical phase $\{p\st\theta(p)>0\}$ by means of sequences of conditions about finite subconfigurations; the definition and properties of the more involved characterization are developed in the next subsection. 

\begin{lemma}\label{lem:subcrit_open}
The set of subcritical parameters $\{p\in(\frac14,1)\st\theta(p)=0\}$ is open. 
\end{lemma}

\begin{lemma}\label{lem:supercrit_open}
The set of supercritical parameters $\{p\in(\frac14,1)\st \theta(p)>0\}$ is open.
\end{lemma}

\begin{proof}[{Proof of Proposition~\ref{prop:connectivity}}]
As a conclusion of the above lemmas, the set $S=\{p\in(\frac14,1)\st\theta(p)=0\}$ is both open and closed in $(\frac14,1)$. By connectivity of this interval, it follows that either $S=(\frac14,1)$ or $S=\emptyset$. Since we already know (cf.~\cite{sidoravicius-tournier}) that $S\subset(\frac14,\frac13)$, we deduce that $S=\emptyset$. 
\end{proof}

\begin{proof}[{Proof of Lemma~\ref{lem:subcrit_open}}]
We have $q=\limup_k q_k$ where, for all $k\in\N$, 
\[q_k=\P((0\leftarrow\bullet)_{[x_1,x_k]}),\]
which gives, using Proposition~\ref{prop:dichotomy},
\begin{align*}
\{p\in(\tfrac14,1)\st\theta(p)=0\}
	& =\{p\in(\tfrac14,1)\st q=1\}\\
	& =\Bigl\{p\in(\tfrac14,1)\st q>\frac1{\sqrt p}-1\Bigr\}\\
	& =\bigcup_{k\in\N}\Bigl\{p\in(\tfrac14,1)\st q_k>\frac1{\sqrt p}-1\Bigr\}.
\end{align*}
On the other hand, each $q_k$ depends only on a configuration of $k$ particles, hence by conditioning on the velocities of these particles we see that $q_k$ is a polynomial in $p$ and therefore is continuous. The lemma follows.
\end{proof}

\begin{proof}[{Proof of Lemma~\ref{lem:supercrit_open}}]
Using the notation $N_k$ from the next subsection, the upcoming Proposition~\ref{prop:characterization} gives
\[\bigl\{p\in(\tfrac14,1)\st\theta(p)>0\bigr\}=\bigcup_{k\in\N}\bigl\{p\in(\tfrac14,1)\st\E[N_k]>0\bigr\}, \]
so that the lemma follows by noticing that, as can be seen by conditioning on the velocities of the $k$ particles, the function $p\mapsto\E[N_k]$ is polynomial hence continuous.
\end{proof}

\subsection{Characterization of the supercritical phase}\label{sec:characterization}

While Lemma~\ref{lem:subcrit_open} relies on the simple monotone approximation $q=\limup_k q_k$, where for all $k\in\N$ the probabilities $q_k=\P((0\leftarrow\bullet)_{[x_1,x_k]})$ depend only on a configuration of $k$ particles, Lemma~\ref{lem:supercrit_open} relies on a formally similar but more involved characterization. This characterization is already alluded to in the first of the final remarks of~\cite{sidoravicius-tournier} as a way to numerically upper bound $\pc{}$. Given its importance in the present proof, we give it here a more thorough presentation, and show it is necessary and sufficient. 

For all $k\in\N$, consider a random configuration containing only the $k$ particles $\bullet_1,\ldots,\bullet_k$ (initially located at $x_1,\ldots,x_k$), and denote by $N_k$ the difference between the number of surviving static particles and the number of surviving left-going particles: letting $I_k=[x_1,x_k]$, this amounts to letting
\[N_k\defeq\sum_{i=1}^k(\indic_{\{\stay_i\}}-
\rlap{\phantom{11}\raisebox{1.5\height}
{\scalebox{.4}{\ensuremath{\leftarrow}}}}
\indic_{\{\bullet_i\}})\indic_{\{\bullet_i\text{ survives}\}_{I_k}}.\]

\begin{proposition}\label{prop:characterization}
For all $p\in(0,1)$, $\theta(p)>0$ $\Leftrightarrow$ $\exists k\ge1,\ \E[N_k]>0$. 
\end{proposition}

\paragraph{\bf Remark.} The fact that $\E[N_1]=\frac12(3p-1)$ recovers (cf.~\cite{sidoravicius-tournier}) that $\theta(p)>0$ when $p>1/3$. The proof of this fact in~\cite{sidoravicius-tournier} is in fact the scheme for the general one given below. Considering $\E[N_2]$ gives the same condition, however $\E[N_3]=3p^3+7p^2\pb-\frac32p\pb^2-8\pb^3$ (where $\pb=\frac{1-p}2$) yields the value $0.32803$ from the remark in~\cite{sidoravicius-tournier}. As the proposition shows, pushing this method further would give arbitrarily good numerical approximations of $\pc{}$. Let us remind that, although such approximations are rendered pointless by Theorem~\ref{thm:main}, the \textit{existence} of this method still is a theoretical tool in the proof of the said theorem. 

\begin{proof}
\noindent{\it Direct implication.} 
Assume that $\theta(p)>0$. Let us decompose $N_k=\Nstay_k-\Ncome_k$, where $\Nstay_k$ and $\Ncome_k$ respectively denote the number of static and left-going particles among $\bullet_1,\ldots,\bullet_k$ that survive in restriction to $[x_1,x_k]$. 

For any integer $i$, the event $\{\stay_i\text{ survives}\}_I$ decreases with the interval $I$ (containing $x_i$). If indeed $\bullet_i$ is static and is annihilated by a particle inside an interval $I$, then introducing new particles outside $I$ can possibly change the side from which $\bullet_i$ is hit, but not the fact that this particle is hit. In particular, the number of static particles among $\bullet_1,\ldots,\bullet_k$ that survive in restriction to $[x_1,x_k]$ is larger than or equal to the number of such particles that survive without restriction, and \latin{a fortiori} to the number of such particles that survive when the initial locations are extended to the full line. 
Taking expectations gives, by shift invariance of the full line process,
\[\E[\Nstay_k]\ge\E_\R[\Nstay_k]= k \P_\R\bigl(\stay_0\text{ survives}\bigr)=kp\theta(p),\]
hence in particular $\E[\Nstay_k]\to+\infty$ as $k\to\infty$. 

On the other hand, $\E[\Ncome_k]$ is uniformly bounded in $k$. Indeed, $\Ncome_k$ clearly grows with $k$, and its limit $\Ncome_\infty=\limup_k \Ncome_k$ is the number of surviving left-going particles in $(0,\infty)$, and this number has geometric distribution with parameter $1-q>0$ (notice indeed that the configuration on the right of a surviving left-going particle is identically distributed to the configuration on $(0,\infty)$, up to translation) and therefore is integrable. 

We conclude that $\E[N_k]=\E[\Nstay_k]-\E[\Ncome_k]\ge kp\theta(p)-\frac q{1-q}\to +\infty$ as $k\to\infty$, hence $\E[N_k]>0$ for large~$k$. 

\vspace{\baselineskip}\noindent{\it Reverse implication.}
Assume now that $\E[N_k]>0$ for some $k\ge1$. 

For positive integers $i<j$, define $N(i,j)$ in the same way as $N_k$ except that only the particles $\bullet_i,\ldots,\bullet_j$ are considered instead of $\bullet_1,\ldots,\bullet_k$. With this notation, $N_k=N(1,k)$. This function $N$ satisfies ``almost'' a superadditivity property. 

\begin{lemma}\label{lem:superadditivity}
Let $k<l$ be positive integers. For any configuration $\omega$ which, in restriction to $[x_1,x_k]$, has no surviving right-going particle, we have
\[N(1,l)\ge N(1,k)+N(k+1,l).\]
\end{lemma}

\begin{proof}[{Proof of Lemma~\ref{lem:superadditivity}}]
For all $k$, let us denote by $\Ngo_k$ the number of right-going particles among $\bullet_1,\ldots,\bullet_k$ that survive in restriction to $[x_1,x_k]$. Observe that the assumption in the statement is that $\Ngo_k=0$. 
Waiving this assumption, we shall prove, by induction on $l-k$, that a slightly more general statement holds: for $0\le k\le l$, \[N(1,l)\geq N(1,k)+N(k+1,l)-\Ngo_k.\] 
Note that we allow the cases $k=0$ (i.e., the interval $[x_1,x_k]$ is empty) and $l=k$ (i.e., the interval $[x_{k+1},x_l]$ is empty), for both of which the result is trivially true.

Suppose there is at least one rightmoving particle among $\bullet_1,\ldots,\bullet_k$ that survives in restriction to $[x_1,x_k]$, and consider the rightmost such particle. If it also survives in restriction to $[x_1,x_l]$ then necessarily all particles among $\bullet_{k+1},\ldots\bullet_l$ that survive in restriction to $[x_{k+1},x_l]$ were right-moving and so we have $N(k+1,l)=0$ and $N(1,l)=N(1,k)$. If not, suppose it annihilates with $\bullet_j$. There are two possibilities: either $\bullet_j$ was the leftmost surviving particle among $\bullet_{k+1},\ldots,\bullet_l$ in restriction to $[x_{k+1},x_l]$ or it was stationary and annihilated by a leftmoving particle on $[x_{k+1},x_l]$. In the former case we have $N(1,j)=N(1,k)$, $\Ngo_j=\Ngo_k-1$ and $N(j+1,l)=N(k+1,l)\pm1$ depending on whether $\bullet_j$ was leftmoving or stationary. In the latter we have $N(1,j)=N(1,k)$, $\Ngo_j=\Ngo_k-1$ and $N(j+1,l)=N(k+1,l)-1$, since the particle which annihilated $\bullet_j$ on $[x_{k+1},x_l]$ survives on $[x_{j+1},x_l]$. In either case the result follows since, by induction, $N(1,l)\ge N(1,j)+N(j+1,l)-\Ngo_j$.

Next, suppose otherwise, i.e.~$\Ngo_k=0$: there is no rightmoving particle among $\bullet_1,\ldots,\bullet_k$ that survives in restriction to $[x_1,x_k]$. If there are no surviving leftmoving particles on $[x_{k+1},x_l]$ then the two sets of particles $\bullet_1,\ldots,\bullet_k$ and $\bullet_{k+1},\ldots,\bullet_l$ do not interact and we have equality. So assume there is at least one surviving leftmoving particle on $[x_{k+1},x_l]$ and consider the leftmost, denoted by $\bullet_j$. If it also survives on $[x_1,x_l]$ then again there are no interactions and we have equality. Otherwise it annihilates with a particle, which must be stationary and was either the rightmost surviving particle on $[x_1,x_k]$ or further right than any such surviving particle. 
In the former case we have $N(1,j)=N(1,k)-1$, $N(j+1,l)=N(k+1,l)+1$ and $\Ngo_j=0$, and in the latter we have $N(1,j)=N(1,k)$, $N(j+1,l)=N(k+1,l)+1$ and $\Ngo_j=1$, so in either case the result follows since, by induction, $N(1,l)\ge N(1,j)+N(j+1,l)-\Ngo_j$. 
\end{proof}

We shall progressively explore the configuration, starting from 0 and going to the right, by repeating the following two steps: first, discover the next $k$ particles, and then discover the least necessary number of particles until there is no surviving right-going particle in the whole discovered region. We will denote by $K_0=0, K_1, K_2,\ldots$, the number of particles discovered in total after each iteration, and by $\Nt^{(1)}(=N_k),\Nt^{(2)},\ldots$ the quantity computed analogously to $N_k$ but on the newly discovered block of $k$ particles at each iteration, i.e., for all $n$, $\Nt^{(n+1)}=N(K_n+1,K_n+k)$. Let us explain the first iteration in some more detail.

We start by considering the first $k$ particles. Let $\Nt^{(1)}=N(1,k)$. If, in the configuration restricted to $[x_1,x_k]$, no right-going particle survives, then we let $K_1=k$. Otherwise, let $\tau_0$ denote the index of the leftmost surviving right-going particle, and appeal for instance to~\cite[Lemma 3.3]{sidoravicius-tournier} to justify the existence of a minimal $\gamma_1$ such that the event $\{\go_{\tau_0}\tto\bullet_{\gamma_1}\}_{[x_{\tau_0},x_{\gamma_1}]}$ happens, and let $K_1=\gamma_1$. By definition we have that, in both cases, in restriction to $[x_1,x_{K_1}]$, there is no surviving right-going particle and $\Nt^{(1)}=N(1,K_1)$. We then keep iterating this construction: define $\Nt^{(2)}=N(K_1+1,K_1+k)$, and keep exploring on the right of $\bullet_{K_1+k}$ until no surviving right-going particle remains, define $K_2$ to be the index that was reached, and so on. By this construction, the random variables $\Nt^{(n)}$ are i.i.d.\ with same distribution as $N_k$, and for all $n$ we have $N(1,K_n+k)=N(1,K_{n+1})$ and there is no surviving right-going particle in restriction to $[x_1,x_{K_{n+1}}]$. Thus, by repeatedly using the lemma, we have for all $n$, 
\[N(1,K_n)\ge \Nt^{(1)}+\cdots+\Nt^{(n)}.\]
However, by the assumption and the law of large numbers, with positive probability $\Nt^{(2)}+\cdots+\Nt^{(n)}>0$ for all $n\ge2$. Therefore, still with positive probability, it may be that the first $k$ particles are static (hence $\Nt^{(1)}=k$) and that $\Nt^{(1)}+\cdots+\Nt^{(n)}> k$ for all $n\ge2$, so that $N(1,K_n)> k$ for all $n\ge2$. This event ensures that 0 is never hit: indeed after the $n$-th iteration of the exploration (for $n\ge2$) there are at least $k+1$ surviving static particles due to the definition of the event, but at most $k$ of them can be annihilated by the particles discovered between $K_n$ and $K_{n+1}$, hence by induction the first static particle survives forever and prevents 0 from being hit. Thus $\theta(p)>0$.
\end{proof}

\section{Proof of Theorem~\ref{thm:universality}: Universality of the law of~$A$}\label{sec:proof_universality}

Remember that $A$ is the index of the first particle, on $\N$, that visits $0$. We wish to prove that the law of $A$ does not depend on $m$ and furthermore, that for all $x\in[-1,1]$, the generating series $w=f(x)=\E[x^A\indic_{\{A<\infty\}}]$ solves the equation
\begin{equation*}
px w^4-(1+2p)xw^2+2w-(1-p)x=0. \tag{\ref*{eq:generating_series}}
\end{equation*}

Since $\P(A=\infty)=1-\sum_{n\in\mathbb N}\P(A=n)$, it is sufficient to show $\P(A=n)$ is independent of the distribution $m$ of interdistances for every $n\in\mathbb N$.
Let, for all $n\in\N$, 
\begin{gather*}
p_n  = \P(A=n),\quad\alpha_n = \P((A=n)\wedge(\stay_1)),\quad\beta_n  =\P((A=n)\wedge(\go_{1}\tto\stay)),\\
\gamma_n  =\P((A=n)\wedge(\go_{1}\collide\come)),\quad\text{and}\quad\delta_n =\P(\go_{1}\collide\come_n).
\end{gather*}
We will prove by induction on $n$ that not only $p_n$ but also $\alpha_n,\beta_n,\gamma_n$ and $\delta_n$ are each independent of the distribution of inter-bullet distances; the proof will also provide recurrence relations from which~\eqref{eq:generating_series} will follow. For $n=1$, we have $p_1=\frac{1-p}2$, and $\alpha_1=\beta_1=\gamma_1=\delta_1=0$. Let $n\in\N$ and assume that the previous property holds up to the value $n-1$. 

First, conditional on the event that $\bullet_1$ is static, $A=n$ if and only if there is some $1<k<n$ such that $\bullet_k$ is the first particle to reach $x_1$ from the right and $\bullet_n$ 
is the first to reach $x_k$ from the right. This happens with a probability equal to $\P(A=k-1)\P(A=n-k)$, which by induction does not dependent on the distribution $m$ for each $k$. Thus 
\begin{equation}\label{eq:alpha}
\alpha_n=p\sum_{1<k<n}p_{k-1}p_{n-k}
\end{equation}
does not depend on $m$. 

Secondly, observe that $\{A=n\}\cap\{\go_{1}\tto\stay_k\}$ occurs if and only if $\bullet_k$ is static, the first particle to reach $x_k$ from the left is $\bullet_1$, 
the first particle to reach $x_k$ from the right is $\bullet_n$, and either $x_k-x_1<x_n-x_k$, or jointly $x_k-x_1=x_n-x_k$ and $s_k=-1$. 
Now, for any configuration~$\omega$,
\[\omega\in\{\stay_k\}\cap\{\go_{1}\fto x_k\}_{(0,x_k)}\cap\{x_k\ffrom\come_n\}_{(x_k,\infty)}\cap\{ x_k-x_1< x_n-x_k\}\] 
if and only if 
\[\rev_n(\omega)\in\{\stay_{k'}\}\cap\{\go_{1}\fto x_{k'}\}_{(0,x_{k'})}\cap\{x_{k'}\ffrom\come_n\}_{(x_{k'},\infty)}\cap\{ x_{k'}-x_1> x_n-x_{k'}\},\]
where $k'=n+1-k$, and similarly
\[\omega\in\{\stay_k\}\cap\{\go_{1}\fto x_k\}_{(0,x_k)}\cap\{x_k\ffrom\come_n\}_{(x_k,\infty)}\cap\{ x_k-x_1= x_n-x_k\}\cap\{s_k=-1\}\] 
if and only if 
\begin{align*}
\rev_n(\omega)\in\{\stay_{k'}\}\cap\{\go_{1}\fto x_{k'}\}_{(0,x_{k'})}\cap\{x_{k'}\ffrom\come_n\}_{(x_{k'},\infty)}&\cap\{ x_{k'}-x_1> x_n-x_{k'}\}\\&\cap\{s_{k'}=+1\}.
\end{align*}

Thus, since $\P$ is invariant under $\rev_n$ (cf.\ Proof of Lemma~\ref{flip}), and $k\mapsto k'$ is a permutation of $\{2,\ldots,n-1\}$, we have
\begin{align*}
\lefteqn{\P((A=n)\wedge(\go_{1}\tto\stay))}\\
	& =\sum_{1<k<n}\P((\stay_k)\wedge(\go_{1}\fto x_k)_{(0,x_k)}\wedge(x_k\ffrom\come_n)_{(x_k,\infty)}\wedge(x_k-x_1<x_n-x_k))\\
	 &+\P((\stay_k)\wedge(\go_{1}\fto x_k)_{(0,x_k)}\wedge(x_k\ffrom\come_n)_{(x_k,\infty)}\wedge(x_k-x_1=x_n-x_k)\wedge(s_k=-1))\\
	& =\sum_{1<k<n}\P((\stay_{k})\wedge(\go_{1}\fto x_{k})_{(0,x_{k})}\wedge(x_{k}\ffrom\come_n)_{(x_{k},\infty)}\wedge(x_{k}-x_1>x_n-x_{k}))\\
	 &+\P((\stay_{k})\wedge(\go_{1}\fto x_{k})_{(0,x_{k})}\wedge(x_{k}\ffrom\come_n)_{(x_{k},\infty)}\wedge(x_{k}-x_1=x_n-x_{k})\wedge(s_{k}=+1))
\end{align*}
hence summing the above two equalities yields exactly
\[2\,\P((A=n)\wedge(\go_{1}\tto\stay))=\sum_{1<k<n}\P\bigl((\stay_k)\wedge(\go_{1}\fto x_k)_{(0,x_k)}\wedge(x_k\ffrom\come_n)_{(x_k,\infty)}\bigr).\]
The events $\{\stay_k\}$, $\{\go_{1}\fto x_k\}_{(0,x_k)}$ and $\{x_k\ffrom\come_n\}_{(x_k,\infty)}$ are independent, and have probabilities
$p$, $\P(A=k-1)$ and $\P(A=n-k)$ respectively, where the expression for the second probability comes from the invariance of $\P$ under $\rev_k$, thus
\[\P((A=n)\wedge(\go_{1}\tto\stay))=\frac p2\sum_{1<k<n}\P(A=k-1)\P(A=n-k),\]
i.e.
\begin{equation}\label{eq:beta}
\beta_n=\frac p2\sum_{1<k<n}p_{k-1}p_{n-k} = \frac12\alpha_n.
\end{equation}
By induction, the terms of this sum do not depend on $m$, so $\beta_n$ doesn't either.

Then, $\P((A=n)\wedge(\go_{1}\collide\come))=\sum_{k=2}^{n-1}\P(\go_{1}\collide\come_k)\P(A=n-k)$, i.e.
\begin{equation}\label{eq:gamma}
\gamma_n=\sum_{1<k<n}\delta_kp_{n-k},
\end{equation}
and by induction the terms of this sum do not depend on $m$. 

Finally, let us consider $\delta_n=\P(\go_{1}\collide\come_n)$. 
If $\{\go_{1}\collide\come_n\}$ occurs then on the interval $(x_1,\infty)$ the first particle to cross $x_1$ is $\bullet_n$, which happens with probability $\P(A=n-1)$. 
However, there are some arrangements where the latter event and $\{\go_{1}\}$ both occur, but $\{\go_{1}\collide\come_n\}$ doesn't. 
In fact, these arrangements are precisely those for which there is some $k<n$ such that $\{A=k\}\cap\{\go_{1}\tto\stay\}$ occurs and $\bullet_n$ is the first to cross $x_k$ from the right. This means that
\[\P(\go_{1}\collide\come_n)=\frac{1-p}2\P(A=n-1)-\sum_{1<k<n}\P((A=k)\wedge(\go_{1}\tto\stay))\P(A=n-k),\]
i.e.
\begin{equation}\label{eq:delta}
\delta_n=\frac{1-p}2p_{n-1}-\sum_{1<k<n}\beta_kp_{n-k},
\end{equation}
and by the induction hypothesis all terms of this sum are independent of the distribution $m$ hence the same holds for $\delta_n$. Since $p_n=\alpha_n+\beta_n+\gamma_n$ when $n\ge2$, this concludes the induction. 

In order to study $f:x\mapsto\sum_{n\ge0}p_n x^n$, let us also define the generating series
\begin{align*}
A:x\mapsto \sum_{n=0}^\infty \alpha_n x^n, \quad
B:x\mapsto  \sum_{n=0}^\infty \beta_n x^n, \quad
C:x\mapsto  \sum_{n=0}^\infty \gamma_n x^n, \quad
D:x\mapsto  \sum_{n=0}^\infty \delta_n x^n.
\end{align*}
Then the previous recurrence relations~\eqref{eq:alpha},\eqref{eq:beta},\eqref{eq:gamma},\eqref{eq:delta} imply respectively
\begin{gather*}
A(x)=px f(x)^2,\qquad B(x)=\frac12A(x),\\
C(x)=D(x)f(x),\quad\text{and } D(x)=\frac{1-p}2xf(x)-B(x)f(x),
\end{gather*}
and since $p_n=\alpha_n+\beta_n+\gamma_n$ when $n\ge2$ and $p_1=\frac{1-p}2$, we conclude that
\begin{align*}
f(x)& =\frac{1-p}2x+A(x)+B(x)+C(x) \\
	& = \frac{1-p}2x+\frac32pxf(x)^2+\frac{1-p}2xf(x)^2-\frac12 pxf(x)^4,
\end{align*}

hence the advertised formula. 

\begin{remark}
We may observe that Formula~\eqref{eq:generating_series} contains sufficient information to imply Theorem~\ref{thm:main} in a more purely algebraic way (see also the end of the introduction), circumventing the topological considerations of Section~\ref{sec:topological}, even though the computations might be hardly tractable in practice. 

The key remark is that $f$ is the only analytic function on the closed unit disk that satisfies~\eqref{eq:generating_series} and $f(0)=0$. Indeed these properties hold for $f$, and the implicit function theorem implies local uniqueness, hence \latin{a fortiori} uniqueness on the disk. Thus, these properties characterize $f$. In particular, they entail the value of $q=f(1)$. One could thus \emph{in principle} deduce the value of $q$ from~\eqref{eq:generating_series}, for any $p\in(0,1)$, without appealing to any \latin{a priori} regularity of $q$. The practical computations, on the other hand, using the expressions of the solutions of quartic equations by radicals, seem to be particularly tedious. 
\end{remark}

\section{Proof of Proposition~\ref{pro:skyline}: Law of the skyline}\label{sec:skyline}

We first prove \ref{p:s:a}). For all $k\ge1$, introduce the $\sigma$-algebra
\[\Fr_k=\sigma((L_i,R_i,\Sigma_i)\st 0\le i\le k).\]
Denote by $S$ the event $\{\bullet\nto\,\stay_0\nfrom\bullet\}$. Let $k\ge1$. 
Let $(l_1,r_1,\sigma_1),\ldots,(l_k,r_k,\sigma_k)$ be any element in the support of the  random variable $(L_1,R_1,\Sigma_1),\ldots,(L_k,R_k,\Sigma_k)$ under $\P_\R(\cdot\s S)$. 

The event $S\cap\{\forall i=1,\ldots,k,\, (L_i,R_i,\Sigma_i)=(l_i,r_i,\sigma_i)\}$ can be decomposed into $\{\bullet\nto\stay_0\}\cap S_k\cap\{x_{r_k}\nfrom\bullet\}$, where $S_k$ is an event that depends \emph{only} on the configuration on $[x_1,x_{r_k}]$. Note in particular that the event $S_k$ implies that no particle among $\bullet_1,\ldots,\bullet_{r_k}$ leaves the interval $[x_1,x_{r_k}]$. 

Then the event $\{ \Sigma_{k+1}=\uparrow \}\cap S\cap\{\forall i=1,\ldots,k,\, (L_i,R_i,\Sigma_i)=(l_i,r_i,\sigma_i)\}$ happens if, and only if the particle at $0$ is static and not hit from the left, the particle at $L_{k+1}(=R_k+1)$ is static and not hit from the right, and the event $S_k$ happens. Thus, by independence and translation invariance properties of the process,
\begin{align*}
\lefteqn{\P_\R((\Sigma_{k+1}=\uparrow)\wedge S\wedge(\forall i=1,\ldots,k,\, (L_i,R_i,\Sigma_i)=(l_i,r_i,\sigma_i)))}\\
& = \P_\R((\bullet\nto\stay_0)_{(-\infty,0]}\wedge S_k\wedge(\stay_{r_k+1})\wedge(x_{r_k+1}\nfrom\bullet)_{[x_{r_k+1},+\infty)})\\
& = \P_\R((\bullet\nto\stay_0)_{(-\infty,0]})\wedge S_k)\,p\,\P((x_{r_k+1}\nfrom\bullet)_{[x_{r_k+1},+\infty)})\\
& = p \P_\R((\bullet\nto\stay_0)_{(-\infty,0]})\wedge S_k\wedge(x_{r_k}\nfrom\bullet)_{[x_{r_k},+\infty)})\\
& = p\P_\R(S\wedge(\forall i=1,\ldots,k,\,(L_i,R_i,\Sigma_i)=(l_i,r_i,\sigma_i))).
\end{align*}
Let $n\ge1$. In the same way, for $\{\Sigma_{k+1}=\nearrow\nwarrow,\ R_{k+1}-L_{k+1}=n\}$, letting $l_{k+1}=r_k+1$ and $r_{k+1}=l_{k+1}+n$, one has
\begin{align*}
\lefteqn{\P_\R((\Sigma_{k+1}=\nearrow\nwarrow)\wedge S\wedge(\forall i=1,\ldots,k+1,\, (L_i,R_i,\Sigma_i)=(l_i,r_i,\sigma_i))}\\
& = \P_\R((\bullet\nto\stay_0)_{(-\infty,0]}\wedge S_k\wedge(\go_{l_{k+1}}\collide\come_{r_{k+1}})\wedge(x_{r_{k+1}+1}\nfrom\bullet)_{[x_{r_{k+1}+1},+\infty)})\\
& = \P_\R((\bullet\nto\stay_0)_{(-\infty,0]})\wedge S_k)\P_\R(\go_{l_{k+1}}\collide\come_{r_{k+1}})\P_\R((x_{r_{k+1}+1}\nfrom\bullet)_{[x_{r_{k+1}+1},+\infty)})\\
& = \P(\go_1\collide\come_{n+1}) \P_\R((\bullet\nto\stay_0)_{(-\infty,0]})\wedge S_k\wedge(x_{r_k}\nfrom\bullet)_{[x_{r_k},+\infty)})\\
& = \P(\go_1\collide\come_{n+1}) \P_\R(S\wedge(\forall i=1,\ldots,k,\,(L_i,R_i,\Sigma_i)=(l_i,r_i,\sigma_i))),
\end{align*}
and similarly for the two other cases.  
Since this holds for all $(l_i,r_i,\sigma_i)$, $i=1,\ldots,k$, it follows that $\{\Sigma_{k+1}=\uparrow\}$ and $\{\Sigma_{k+1}=\nearrow\nwarrow, L_{k+1}-R_{k+1}=n\}$ are independent of $\Fr_k$ under $\P_R(\cdot\s S)$ and have respective probabilities $p$ and $\P(\go_1\collide\come_{n+1})$. 
Note that these events belong to $\Fr_{k+1}$. Since this holds for all $k\ge1$, the law of $(\Sigma_k,\Delta_k)_{k\ge1}$ follows, hence the statement of the proposition for positive $k$. The identity $\P(\Sigma=\nearrow\nwarrow)=\P(\go_1\collide\come)$ is obtained by summation over $n$, and the explicit value $(1-p)/2-r-s=\sqrt p-p$ was obtained for instance in the course of the proof of Lemma~\ref{lem:identity}.

Finally, given $S$, the particles on $(0,+\infty)$ and $(-\infty,0)$ are independent and have distributions symmetric to each other, hence the conclusion.

The statement of \ref{p:s:b}) then follows at once from the distribution of $\Sigma$. 

\section{Asymptotics of densities}\label{sec:asymptotics}

The implicit equation~\eqref{eq:generating_series} enables us, by analytic combinatorial methods, to compute asymptotics of the distribution of $A$, the index of the leftmost particle to ever visit 0 (Theorem~\ref{thm:asymptotics}). We shall then deduce asymptotics of the density of surviving particles as time passes (Theorem~\ref{thm:asympt_density}). The transition from the law of $A$ to the law of the lifetime of a static particle, i.e.\ its time of first collision, either from right or left, assumes however some control of the distance between particles (Lemmas~\ref{lem:discrete_continuous},\ref{lem:discrete_continuous_super}). 

Let us first state the result on $A$, then deduce its consequence, and finally return to the proof of this result in the last part of this section.

\begin{theorem}\label{thm:asymptotics}
We have
\[\E[A]=\infty\qquad\text{if $p\le1/4$,}\]
and
\[\E[A\mid A<\infty]=\frac{\sqrt p}{(2\sqrt p-1)(1-2\sqrt p+2p)}\qquad\text{if $p>1/4$.}\]
Also, 
\[\P(A=n)\equivalent{\substack{n\to\infty\\n\text{ odd}}}\begin{cases}
\frac{\sqrt 2}{\sqrt{\pi(1-4p)}}n^{-3/2}	& \text{if $p<1/4$,}\\
\frac{2^{4/3}}{3\Gamma(2/3)}n^{-4/3}	& \text{if $p=1/4$,}\\
\frac13\sqrt{\frac{2(8p+1)(1-p)}{\pi p(4p-1)}} R^{-n}n^{-3/2}	&\text{if $p>1/4$,}\\
\end{cases}\]
where $R=R(p)=\frac3{8p+1}\sqrt{\frac {3p}{1-p}}$,
and the analogous quantity for a right-going particle at zero satisfies
\[\P_\R(\go_{0}\collide\come_n)
\equivalent{\substack{n\to\infty\\n\text{ odd}}}
\begin{cases}
\frac{\sqrt{1-4p}}{\sqrt{2\pi}}n^{-3/2} & \text{if $p<1/4$,}\\
\frac1{2^{1/3}\Gamma(1/3)}n^{-5/3}	&\text{if $p=1/4$}\\
c R^{-n}n^{-5/2}	& \text{if $p>1/4$,}
\end{cases}
\]
for some $c=c(p)>0$.
\end{theorem}

Let us consider the process defined on the full line. 
Denote respectively, for all $t>0$, by $c_0(t)$ and $c_{+}(t)$ the density of static particles and of right-going particles that have not annihilated by time $t$, among the particles present at initial time, i.e.\ 
\[c_0(t)=\lim_{m,n\to+\infty} \frac1{m+n+1}\sum_{-m\le k\le n}\indic_{\{\stay_k\}\cap\{\bullet_k\text{ survives beyond time $t$}\}},\]
and similarly for $c_+(t)$ with $\{\go_k\}$ instead of $\{\stay_k\}$. 
These densities exist a.s.\ due to the ergodicity of the process under $\P_\R$, and satisfy
\[c_+(t)=\P_\R(\go_0,\ \text{$\bullet_0$ survives beyond time $t$})=\P_\R(\go_{0}\tto t)\]
and
\[c_0(t)=\P_\R(\stay_0,\ \text{$\bullet_0$ survives beyond time $t$})=p\P(D>t)^2,\]
where the final equality comes from the fact that the survival of a static particle at 0 beyond time $t$ is equivalent to the survival from particles starting in $(0,t]$ and in $[-t,0)$, together with symmetry and independence of both half-lines.

Note that, if $m$ is integrable with mean 1, then $c_0(t)$ and $c_+(t)$ also have the meaning of spatial densities: by the law of large numbers, 
\[c_0(t)=\lim_{\substack{a\to\infty\\b\to\infty}} \frac1{a+b}\#\{k\in\Z\st x_k\in[-a,b],\ \bullet_k\text{ is static and survives after time $t$}\},\]
and similarly for $c_+(t)$. 

The previous critical and subcritical polynomial asymptotics reinterpret into universal asymptotics for $c_0(t)$ and $c_+(t)$ as $t\to\infty$: recall the following result from the introduction.

{\renewcommand{\thetheorem}{\ref*{thm:asympt_density}}
\begin{theorem}[Asymptotics of the density of particles]Assume the law $m$ of distance between particles to be exponentially integrable (i.e.\ $\int e^{\eta x}dm(x)<\infty$ for some $\eta>0$) and have mean $1$. Then, for some $c=c(p)>0$, as $t\to\infty$, 
\[c_0(t)=\begin{cases}
\bigl(\frac{2p}{\pi(1-4p)}+o(1)\bigr)\,t^{-1}	& \text{if $p<1/4$,}\\
\bigl(\frac{2^{2/3}}{4\Gamma(2/3)^2}+o(1)\bigr)\,t^{-2/3}	& \text{if $p=1/4$,}\\
(2\sqrt p-1)^2-o(e^{-ct})	&\text{if $p>1/4$,}
\end{cases}\]
and
\[c_+(t)=\begin{cases}
\bigl(\frac1{\sqrt\pi}\sqrt{1-4p}+o(1)\bigr)\,t^{-1/2}	&\text{if $p<1/4$,}\\
\bigl(\frac{2^{2/3}}{8\Gamma(2/3)^2}+\frac{3 }{8\Gamma(1/3)}+o(1)\bigr)\,t^{-2/3}	&\text{if $p=1/4$,}\\
o(e^{-ct})	&\text{if $p>1/4$.}
\end{cases}\]
Furthermore, when $m=\delta_1$, if $p>1/4$, 
\begin{equation}
(2\sqrt p-1)^2-c_0(n)\equivalent{\substack{n\to\infty\\n\text{ even}}}2c_+(n) \equivalent{\substack{n\to\infty\\n\text{ even}}}\frac{9p}{2\sqrt p+1}\sqrt{\frac{8(1-p)(8p+1)}{\pi(4p-1)^5}}R^{-(n+1)}n^{-3/2}\tag{\ref*{eq:equiv_c0_discrete}}
\end{equation}
where $R=R(p)=\frac3{8p+1}\sqrt{\frac {3p}{1-p}}$.
\end{theorem}
\addtocounter{theorem}{-1}}

This result is a consequence of the previous asymptotics (Theorem~\ref{thm:asymptotics}) and of the following Lemmas~\ref{lem:discrete_continuous} and~\ref{lem:discrete_continuous_super}, which rely on standard large deviations estimates to control the approximation of the distance $D=x_A$ by the index $A$. 

In the statements below, we denote by $B$ the index (in $\Z$) of the particle that collides with the particle at $0$, and $B=\infty$ if no such particle exists. 

\newcommand{\floor}[1]{\left\lfloor #1 \right\rfloor}
\begin{lemma}\label{lem:discrete_continuous}
Assume $m$ is exponentially integrable and has unit mean. Assume $0<p\le 1/4$. Let $C,\alpha>0$. 
\begin{enumerate}[a)]
	\item 
If
\[\P(A=n)\equivalent{\substack{n\to\infty\\n\text{ odd}}} Cn^{-(1+\alpha)},\] 
then
\begin{equation}
\P(D>t)\equivalent{t\to\infty}\P(A>\floor{t})\equivalent{t\to\infty} \frac C{2\alpha} t^{-\alpha}\label{eq:DA}
\end{equation}
and
\begin{equation}
\P_\R(\go_{0}\tto\stay,\ x_B>t)\equivalent{t\to\infty} p\frac{C^2}{8\alpha^2} t^{-2\alpha}.\label{eq:equivGo}
\end{equation}
	\item
If 
\[\P_\R(\go_{0}\collide\come_n)\equivalent{\substack{n\to\infty\\n\text{ odd}}} Cn^{-(1+\alpha)},\]
then
\begin{equation}
\P_\R(\go_{0}\collide\come,\ x_B>t)\equivalent{t\to\infty} \frac C{2\alpha} t^{-\alpha}.\label{eq:equivB}
\end{equation}
\end{enumerate}
\end{lemma}

\begin{lemma}\label{lem:discrete_continuous_super}
Assume $m$ is exponentially integrable and has unit mean. Assume $1/4<p<1$. Then we have the following (at least) exponential decays:
\begin{align*}
\P(t<D<\infty) & = o(e^{-\eta t}),\\
\P_\R(\go_{0}\tto\stay,\ x_B>t) & =o(e^{-\eta t}),\\
\P_\R(\go_{0}\collide\come,\ x_B>t) & =o(e^{-\eta t}),
\end{align*}
for some $\eta=\eta(p)>0$. 
\end{lemma}

Finally, the following simple lemma is easily established by standard methods and is therefore merely mentioned for reference:
\begin{lemma}
For any $\alpha>0$, 
\begin{equation}\label{eq:equiv_polynomial}
\sum_{\substack{k\ge n\\k\text{ odd}}}\frac1{k^{1+\alpha}}\equivalent{n\to\infty}\frac1{2\alpha}{n^{-\alpha}}.
\end{equation}
For any $\alpha\in\R$ and $R>1$, 
\begin{equation}\label{eq:equiv_exponential}
\sum_{\substack{k\ge n\\k\text{ odd}}}\frac1{k^\alpha R^k}\equivalent{\substack {n\to\infty\\\text{$n$ odd}}}\frac{R^2}{R^2-1}\frac1{n^\alpha R^n}.
\end{equation}
\end{lemma}

\begin{proof}[Proof of Lemma~\ref{lem:discrete_continuous}]
Let us first prove \eqref{eq:DA}. Note that the second asymptotic equivalent follows from the polynomial decay assumption by~\eqref{eq:equiv_polynomial} and by the fact that $A<\infty$ a.s.\ if $p\le1/4$ (Theorem~\ref{thm:main}). Therefore we only have to prove the first asymptotics. 

Remember that $D=x_A=X_1+\cdots+X_A$. 
Then, for any $\delta>0$, for all $t>0$, since $X_i\ge0$ $\forall i$,
\[\{A<\floor{(1-\delta)t}\}\cap\{D>t\}\subset\{X_1+\cdots+X_{\floor{(1-\delta)t}}>t\},\]
and
\[\{A>\floor{(1+\delta)t}\}\cap\{D<t\}\subset\{X_1+\cdots+X_{\floor{(1+\delta)t}}<t\},\]
and by classical large deviation principles the probabilities of both right hand sides decay (at least) exponentially fast as $t\to\infty$ under the assumption of exponential integrability of $m$. It follows that $\P(D>t)$ is bounded above by $\P(A>\floor{(1-\delta)t})$ and below by $\P(A>\floor{(1+\delta)t})$ up to exponentially small terms. As mentioned above, the assumption on $A$ and $p$ implies elementarily that $P(A>n)\sim C/(2\alpha)n^{-\alpha}$, hence  
\[(1+\delta)^{-\alpha}\le\liminf_{t\to\infty}\frac{\P(D>t)}{\frac C{2\alpha}t^{-\alpha}}\le\limsup_{t\to\infty}\frac{\P(D>t)}{\frac C{2\alpha}t^{-\alpha}}\le(1-\delta)^{-\alpha},\]
and~\eqref{eq:DA} follows by letting $\delta\to0$.

The proof of~\eqref{eq:equivB} from the assumption in b) is obtained in the same way as~\eqref{eq:DA} above. 

Let us turn to~\eqref{eq:equivGo}. It suffices for us to prove
\begin{equation}
\P_\R(\go_{0}\tto\stay_n)\equivalent{\substack{n\to\infty\\\text{$n$ odd}}} p\P(A=n)\P(A>n),\label{eq:234}
\end{equation}
because by the assumption this implies $\P_\R(\go_{0}\tto\stay_n)\sim pC^2/(2\alpha)n^{-(1+2\alpha)}$, and~\eqref{eq:equivGo} will then again follow by the same approximation from discrete to continuous as above.
First notice that, by the same arguments as in the proof of the identity~\eqref{r-relation} in Lemma~\ref{flip},
\begin{equation}
\P_\R(\go_{0}\tto\stay_n) =p\P(A=n,\ D'>D)+\frac12 p\P(A=n,\ D'=D),\label{eq:123}
\end{equation}
where $D'$ has same distribution as $D$ and is independent of $A$ and $D$. Then, as before, for any $\delta>0$, for some $\gamma(=\gamma(\delta))>0$, classical large deviations give $\P(A=n, D\le(1-\delta)n)=o(e^{-\gamma n})$, hence
\begin{align*}
\P(A=n,D'>D) 
	& = \P(A=n, D>(1-\delta)n, D'>D)+o(e^{-\gamma n})\\
	& \le\P(A=n, D'>(1-\delta)n)+o(e^{-\gamma n}) \\
	& = \P(A=n)\P(D>(1-\delta)n)+o(e^{-\gamma n})\\
	& = \P(A=n)\P(A>\floor{(1-\delta)n})+o(n^{-2\alpha-1})\\
	& = (1-\delta)^{-\alpha}\P(A=n)\P(A>n)+o(n^{-2\alpha-1}).
\end{align*}
We proceed similarly for the lower bound:
\begin{align*}
\P(A=n,D'>D) 
	& = \P(A=n, D<(1+\delta)n, D'>D)+o(e^{-\gamma n})\\
	& \ge \P(A=n, D<(1+\delta)n<D')+o(e^{-\gamma n})\\
	& =\P(A=n, D<(1+\delta)n)\P(D>(1+\delta)n)+o(e^{-\gamma n})\\
	& = \P(A=n)\P(D>(1+\delta)n)+o(e^{-\gamma n})\\
	& = (1+\delta)^{-\alpha}\P(A=n)\P(A>n)+o(n^{-2\alpha-1}),
\end{align*}
and conclude that $\P(A=n,\,D'>D)\sim\P(A=n)\P(A>n)$ by letting $\delta\to0$ as for~\eqref{eq:DA}. The same proof shows that the right-hand side is also equivalent to $\P(A=n,D'\ge D)$. Since, by~\eqref{eq:123},
\[p\P(A=n, D'>D)\le \P_\R(\go_0\tto\stay_n)\leq p\P(A=n, D'\ge D),\]
we conclude that \eqref{eq:234} holds, which entails~\eqref{eq:equivGo}.
\end{proof}

\begin{remark}\label{rmk:exponential_integrability}
As shown in the previous proof, the exponential integrability of $m$ ensures that $D$ and $A$ have matching asymptotics, which could fail otherwise. Assuming for instance that $\P(x_1>t)\sim C t^{-r}$ for some $r>1$, one has polynomial large deviations $\P(x_n>(1+\delta)n)\sim Cn^{-r+1}$ (cf.~\cite{nagaev81} or more recently~\cite[Theorem 9.1]{denisov08}) so that for instance
\begin{align*}
\P(D>t)
	& \ge\P(X_1+\cdots+X_{\floor{t/2}}>t,\ A<t/2)\\
	& \ge \P(X_1+\cdots+X_{\floor{t/2}}>t)-\P(A>t/2),
\end{align*}
and, provided $r$ is close enough to $1$, the leading order of the right hand side would be $t^{-r+1}$, making $\P(D>t)$ larger that in the universal case. 
\end{remark}

\begin{proof}[Proof of Lemma~\ref{lem:discrete_continuous_super}]
Choose any $\delta\in(0,1)$. Then, similarly as in the proof of the previous lemma, 
\begin{align*}
\P(t<D<\infty)
	& \le \P(\floor{(1-\delta)t}<A<\infty)+\P(X_1+\cdots+X_{\floor{(1-\delta)t}}>t)
\end{align*}
and both quantities on the right hand side decay (at least) exponentially fast as $t\to\infty$. The first assertion of the lemma follows. 

The third one is obtained in the same way. 

Finally, one still has~\eqref{eq:123} (with $D'$ now possibly taking the value $\infty$) hence in particular
\[\P_\R(\go_{0}\tto\stay_n)\le \P(A=n)=o(e^{-\eps n})\]
for some $\eps>0$. This exponential decay in the discrete index $n$ is then turned into an exponential decay of $\P_\R(\go_{0}\tto\stay,\,x_B>t)$ in the continuous variable $t$ in the same way as above. 
\end{proof}

\begin{proof}[Proof of Theorem~\ref{thm:asympt_density}]
Recall
\[c_0(t)=p\P(D>t)^2 = p(1-q+\P(t<D<\infty))^2,\]
hence the asymptotics for $c_0$ follow at once from Theorems~\ref{thm:main} and~\ref{thm:asymptotics}, Lemma~\ref{lem:discrete_continuous} (Relation~\eqref{eq:DA}) and Lemma~\ref{lem:discrete_continuous_super}. 

As for $c_+$, notice that, if the particle at $0$ is right-going, then it reaches location $t$ (or, equivalently, survives until time $t$) if and only if it meets no static particle launched from $(0,t]$ and no left-going particle launched from $(0,2t]$. Thus 
\[c_+(t)=\P_\R(\go_{0}\tto t) = \P_\R(\go_{0}\tto\stay,\ x_B>t)+ \P_\R(\go_{0}\collide\come,\ x_B>2t).\]
The asymptotics for $c_+$ now follow from Theorem~\ref{thm:asymptotics}, Lemma~\ref{lem:discrete_continuous} (Relations~\eqref{eq:equivGo} and~\eqref{eq:equivB}), and Lemma~\ref{lem:discrete_continuous_super}. Note that, for $p<1/4$, the first term (collision with a static particle) decays as $t^{-1}$ hence is negligible with respect to the second term decaying as $t^{-1/2}$; that, for $p=1/4$, both terms are of the same order ; and that, for $p>1/4$, both terms are exponentially small. 

Let us finally consider the case $m=\delta_1$, $p>1/4$. Since $D=A$ in this case, the asymptotics for $c_0(n)$ come from those of $A$. We indeed have as above, as $n\to\infty$,
\[c_0(n)=p (1-q+\P(n<A<\infty))^2 = p(1-q)^2+2p(1-q)\P(n<A<\infty)(1+o(1)),\]
hence~\eqref{eq:equiv_c0_discrete}, using~\eqref{eq:equiv_exponential} (and computing in particular $\frac{R^2}{R^2-1}=\frac{27p}{(4p-1)^3}$). 
For $c_+(n)$, first remember
\[c_+(n)=\P_\R(\go_0\tto n)=\P_\R(\go_0\tto\stay, B>n)+\P_\R(\go_0\collide\come, B>2n).\] 
The second term is obtained by summations of $\P(\go_0\collide\come_k)$, $k>2n$, hence by Theorem~\ref{thm:asymptotics} and~\eqref{eq:equiv_exponential} is of order $C(p)R^{-2n}n^{-5/2}$ for some constant $C(p)$. For the first term, let us first, as above, write
\begin{align*}
\P_\R(\go_0\tto\stay_k) &=  p\P(A=k, A'>k)+\frac12 p\P(A=k, A'=k) \\
	& = p\P(A=k)\P(A>k)+\frac12 p\P(A=k)^2.
\end{align*}
The second summand is of smaller order as $k\to\infty$, while the first is equivalent to $p\P(A=k)\P(A=\infty)=p(1-q)\P(A=k)$. Thus, by summation over $k>n$, and comparison with the above formula for $c_0(n)$, as $n\to\infty$, 
\[\P_\R(\go_0\tto\stay, B>n)\sim p(1-q)\P(n<A<\infty)\sim \frac12\bigl(p(1-q)^2-c_0(n)\bigr) \sim C' R^{-n}n^{-3/2},\]
for some $C'>0$. In particular this term dominates $\P_\R(\go_0\collide\come, B>2n)$ (which is asymptotically equivalent to $C(p)R^{-2n}n^{-5/2})$, hence the result. 
\end{proof}

Let us finally prove the main result of this section. 

\begin{proof}[Proof of Theorem~\ref{thm:asymptotics}]
\textit{Expectation of $A$.} We know \latin{a priori} that $\E[A\indic_{\{A<\infty\}}]<\infty$ when $p>1/4$. Indeed, the set of indices of surviving static particles on the full-line is ergodic, hence the first positive index $S$ of a surviving static particle is integrable; and we have $S>A$ on $\{A<\infty\}$, hence $\E[A\indic_{\{A<\infty\}}]\le\E_\R[S\indic_{\{A<\infty\}}]\le\E_\R[S]$ is finite. 
Alternatively, this would follow from the fact, proved below, that $f$ is analytic hence differentiable in a neighborhood of $1$ when $p>1/4$. 

Since $\E[A\indic_{\{A<\infty\}}]=f'(1^{-})$, the first formulas follow by differentiating Equation~\eqref{eq:generating_series} with respect to $x$, and substituting $x=1$ (hence $f(x)=f(1)=q$), yielding, if $p>1/4$, and using the above fact that $\E[A\indic_{\{A<\infty\}}]<\infty$,
\[\E[A\indic_{\{A<\infty\}}]=\frac q{(1-q)(1-2pq)}.\]
The expression for $q$ obtained in Theorem~\ref{thm:main} leads to the stated formula. 

When $p\le1/4$, if we assume $\E[A]<\infty$ then this same procedure (now with $q=1$) produces a contradiction, hence $\E[A]=\infty$.

\vspace{\baselineskip}\noindent\textit{Analytic continuation.} We shall prove that the function $f$ is amenable to the methods of singularity analysis (cf.~\cite[Chapter VI]{flajolet-sedgewick}), that enables one to deduce asymptotics of the coefficients $p_n=\P(A=n)$, $n\in\N$, of $f$ from asymptotics of the analytic function $f$ in the neighbourhood of its singularities of least modulus. To that end, let us first discuss the analytic continuation of $f$. 

As the generating series of a (sub-)probability distribution, $f$ is convergent in the unit disk, hence is analytic in this disk, and by Theorem~\ref{thm:universality} satisfies the identity~\eqref{eq:generating_series} in this disk, i.e., for all $z\in\C$ such that $\lvert z\rvert\le1$, \[F(z,f(z))=0,\] 
where the polynomial $F$ in two variables is defined by
\begin{equation}\label{eq:F}
F:(z,w)\mapsto F(z,w)=pzw^4-(1+2p)zw^2+2w-(1-p)z.
\end{equation}

For a given $z\in\C$, the equation $F(z,w)=0$ is solved by up to 4 complex numbers $w$, one of which, when $\lvert z\rvert\le 1$, is $f(z)$. Furthermore, wherever the local inversion theorem applies, i.e.\ for all $(z_0,w_0)\in\C^2$ such that $F(z_0,w_0)=0$ and $\partial_w F(z_0,w_0)\ne 0$, there is a unique holomorphic function $\phi$ defined on a neighbourhood $U_0$ of $z_0$ and such that $\phi(z_0)=w_0$ and $F(z,\phi(z))=0$ for all $z\in U_0$. 

The finitely many couples $(z_0,w_0)$ such that $F(z_0,w_0)=0$ and conversely $\partial_w F(z_0,w_0)=0$ are known as \emph{singularities} of $F$. By general arguments (see for instance \cite[Chapter 2, Section 1, Theorem 2]{siegel}), $f$ can be analytically continued along each path that avoids these singularities hence, by the monodromy theorem, can be analytically continued to any simply connected domain of the complex plane that does not contain these singularities (typically, a slit plane). One needs however a finer analysis in order to identify the radius of convergence of $f$ and the singularities on the boundary of its disk of convergence, as some of the singularities of $F$ are singularities of other analytic solutions $\phi$ to $F(z,\phi(z))=0$. (Actually, by irreducibility of $F$, all such solutions can be obtained by analytic continuations of $f$, see~\cite[Chapter 2, Section 1, Theorem 3]{siegel}, a property best stated within the framework of Riemann surfaces.) Regarding this problem, we refer to Figure~\ref{fig:riemann} for an illustration, and to Reference~\cite{canfield84} for an instructive discussion. 

\begin{figure}\label{fig:riemann}
\begin{center}
\includegraphics[width=10cm]{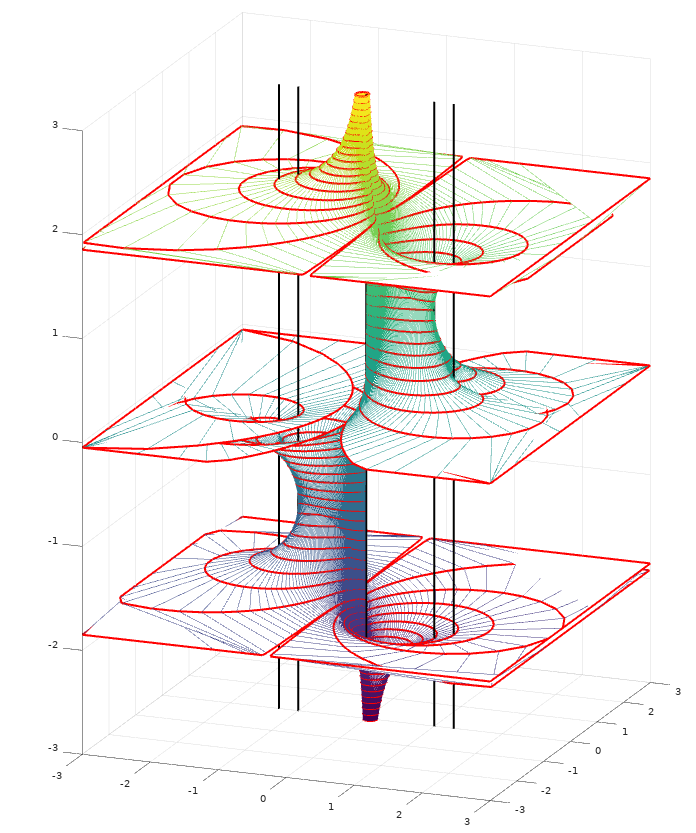}
\end{center}
\caption{Plot of the real part of the solutions $w\in\mathbb C$ to $F(z,w)=0$ (see~\eqref{eq:F}), as a function of $z\in\mathbb C$ (horizontal plane), for $p=0.75$. Vertical lines correspond to the values $z=\pm1, \pm R$ (here, $R>1$). On this representation, a slice corresponding to real parts in $[-0.05, +0.05]$ was removed, so as to distinctly show the two sheets in the middle part, the one closer to the plane $\Re(w)=0$ being the real part of $f$. Note in particular that $1$ is a singularity of $F$, although not of $f$. A few level curves (truncated on the boundary) are shown in red.} 
\end{figure}

\vspace{\baselineskip}\noindent\textit{Singularities of $F$.} The equation $\partial_w F(z,w)=0$ implies $w\ne0$, $w\ne\sqrt{\frac{1+2p}{2p}}$, and $z=\frac1w\frac1{(1+2p)-2pw^2}$, while the additional constraint $F(z,w)=0$ yields (keeping only the numerator)
\[0=(w^2-1)\bigl(3pw^2-(1-p)\bigr),\]
hence $w=\pm1$ or $w=\pm\sqrt{\frac{1-p}{3p}}$, and respectively $z=\pm1$ or $z=\pm R$ where 
\[R=\frac3{8p+1}\sqrt{\frac{3p}{1-p}}.\] 
The points $(\pm 1,\pm1)$ and $(\pm R, \pm\sqrt{\frac{1-p}{3p}})$ are thus singularities of $F$ --- or, in other words, of the analytic multivalued function obtained by considering all analytic continuations of $f$. 

Let us note that it is more natural to consider more generally $z,w$ in the Riemann sphere $\widehat\C=\C\cup\{\infty\}$, in which case $(0,\infty)$ can be seen to be another singularity of $F$ (indeed, $(0,0)$ is a singularity of $G(z,v)=v^4 F(z,1/v)$). 

\vspace{\baselineskip}\noindent\textit{Identification of the singularities of $f$.} Since $f(0)=0\ne\infty$ (or since obviously $f$ is analytic in a neighborhood of $0$), $0$ is not a singularity of $f$. We first consider the candidate singularities at $(z,w)=\pm(1,1)$. By definition of $f$, we have $f(\pm1)=\pm q$. Thus, by Theorem~\ref{thm:main}, 
\begin{itemize}
	\item if $p\le1/4$, then $q=1$ hence indeed $\pm1$ are singularities of $f$;
	\item if $p>1/4$, then $q<1$ hence $\pm1$ are not singularities of $f$.
\end{itemize}
The other candidate singularities are at $(z,w)=\pm\Bigl(\sqrt{\frac{1-p}{3p}},R\Bigr)$. We have:
\begin{itemize}
	\item if $p<1/4$, then $R\in(0,1)$ hence $f$ is analytic at $R$, i.e.\ $\pm R$ are not singularities for $f$;
	\item if $p=1/4$, then $R=1$ so $\pm R=\pm 1$ are the same singularities as before;
	\item if $p>\frac14$, then $R>1$. Since $\pm1$ are not singularities, $\pm R$ \emph{must} be singularities of $f$. Indeed, the coefficients $p_n$ decay at most exponentially (since, for instance, $p_n\ge((1-p)/2)^n$ for odd $n$ by considering a configuration of alternating $+1$ and $-1$ particles), which implies that $f$ has a finite radius of convergence (e.g., smaller than $2/(1-p)$).  
\end{itemize}

\vspace{\baselineskip}\noindent\textit{Singularity analysis.} We are thus in position to apply singularity analysis (cf.~\cite[Chapter VI]{flajolet-sedgewick}, and more specifically Corollary VI.1 \emph{Sim-transfer} and Theorem VI.5 \emph{Multiple singularities}) in order to obtain asymptotics for $p_n$. 

Assume first $0<p<1/4$. By the previous discussion, the radius of convergence of $f$ is $1$, with symmetric singularities at $\pm1$. We compute, from the equation $F(z,f(z))=0$, and recalling $f(1)=1$, 
\begin{equation}f(z)-1\equivalent{z\to1}-\frac{\sqrt{2}}{\sqrt{1-4p}}(1-z)^{1/2},\label{eq:fz-near-1}\end{equation}
and a symmetric property holds at $-1$ since $f$ is odd (indeed $p_{2n}=0$ for all $n$), hence
\[p_n\equivalent{n\to\infty}-\frac{\sqrt 2}{\sqrt{1-4p}}\frac{n^{-3/2}}{\Gamma(-1/2)}(1-(-1)^n)=\frac{\sqrt 2}{\sqrt{\pi(1-4p)}}n^{-3/2}\indic_{\text{($n$ odd)}}\]

Assume now $p=1/4$. Again the singularities are at $\pm1$. We find similarly
\[f(z)-1\equivalent{z\to1}-2^{1/3}(1-z)^{1/3}.\]
It follows by the same symmetry that
\[p_n\equivalent{n\to\infty}-2^{1/3}\frac{n^{-4/3}}{\Gamma(-1/3)}(1-(-1)^n)=\frac{2^{4/3}}{3\Gamma(2/3)}n^{-4/3}\indic_{\text{($n$ odd)}}.\]

Finally, assume $1/4<p<1$. Then the singularities are at $\pm R$ and $R>1$ so we can conclude that $p_n$ decays exponentially fast as $n\to\infty$, in other words $A$ is exponentially integrable on $\{A<\infty\}$, and more precisely we compute
\[f(z)-f(R)\equivalent{z\to R} -\frac{\sqrt{2}}{3^{7/4}} \frac{8p+1}{\sqrt{4p-1}}\Big(\frac1p-1\Big)^{3/4}(R-z)^{1/2}\]
hence
\[p_n\equivalent{n\to\infty}C_p R^{-n}n^{-3/2}\indic_{\text{($n$ odd)}},\]
where $C_p=\frac1{\sqrt\pi}\frac{\sqrt{2}}{3^{7/4}} \frac{8p+1}{\sqrt{4p-1}}(\frac1p-1)^{3/4}R^{1/2} = \frac{\sqrt 2}{3\sqrt\pi}\Bigl(\frac{(8p+1)(1-p)}{p(4p-1)}\Bigr)^{1/2}$. 

\vspace{\baselineskip}\noindent\textit{Asymptotics for $\delta_n$.} Let us finally find the asymptotics for $\delta_n=\P(\go_{1}\collide\come_n)$ in each regime, which will at once give asymptotics for $\P_\R(\go_{0}\collide\come_n)=\delta_{n+1}$. 

Recall that we have
$A(x)=pxf(x)^2$, $B(x)=\frac12A(x)$, $C(x)=D(x)f(x)$,  
and $D(x)=\frac{1-p}2xf(x)-B(x)f(x)=\frac{1-p}2xf(x)-\frac12pxf(x)^3$, 
hence $A,B,C,D$ have (at most) the same singularities as $f$. 

Assume $p<1/4$. Using \eqref{eq:fz-near-1} we can compute, as $\eps\to0$,
\[D(1-\eps)=\frac12-p-\frac{\sqrt{1-4p}}{\sqrt2}\eps^{1/2}+o(\eps^{1/2})\]
and we deduce 
\[\delta_n\equivalent{\substack{n\to\infty\\n\text{ even}}} \frac{\sqrt{1-4p}}{\sqrt{2\pi}}n^{-3/2}.\]

Assume $p=1/4$. Then, since $1-f(1-\eps)\sim 2^{1/3}\eps^{1/3}$, 
\[\frac14-D(1-\eps)\equivalent{\eps\to0}\frac382^{2/3}\eps^{2/3}\]
hence
\[\delta_n\equivalent{\substack{n\to\infty\\n\text{ even}}}2\times\frac38\frac{2^{2/3}}{\Gamma(-2/3)}n^{-5/3}=\frac{2^{2/3}}{2\Gamma(1/3)}n^{-5/3}.\]

Finally, assume $p>1/4$. 
From finer asymptotics of $f$ we can still find that $D$ has a development of the form
\[D(R)-D(R-\eps)=a\eps+b\eps^{3/2}+o(\eps^{3/2})\]
for some non-zero $a,b\in\R$ (a notable fact is that the \latin{a priori} leading order in $\eps^{1/2}$ vanishes) so that 
\[\delta_n \equivalent{\substack{n\to\infty\\n\text{ even}}} \frac{3bR^{3/2}}{2\sqrt\pi}R^{-n}n^{-5/2}.\]
Note that the proof of Theorem~\ref{thm:asympt_density} does actually not use such precise asymptotics of $\delta_n$ for $p>1/4$, but simply its exponential rate $R^{-n}$, which in the general case shows exponential decay, and in the case $m=\delta_1$ shows that the term involving $\delta_n$ is negligible. 
\end{proof}

\section{Consequences on the discretized model}\label{sec:discrete}

In this section we turn to the discretized model introduced in \cite{junge2}, and the same model extended to other situations where the measure $m$ of interdistances has atoms. Recall that the difference from our model lies in the way that triple collisions are resolved: in \cite{junge2} they result in the destruction of all three particles. If $m$ is atomless, triple collisions almost surely do not arise, and so the two models are equivalent. In this section we shall assume that $m$ has at least one atom, so the two models are not the same. We use $\Ph$ and $\Ph_{\R}$ for the half-line and full-line versions of this model, and write
\begin{gather*}
\qh\defeq\Ph(0\from\bullet),\\
\sh\defeq\Ph{\bigl((0\nfrom\bullet)\wedge(\go_{1}\tto\stay)\bigr)},\\
\rh\defeq\Ph{\bigl((0\from\bullet)\wedge(\go_{1}\tto\stay)\bigr)}
\end{gather*}
for the analogues of the quantities considered in Section~\ref{sec:algebraic}. In contrast to our other results, the behavior of this model is not universal, and we shall therefore concentrate on the specific case studied in \cite{junge2}, where $m=\delta_1$, i.e.\ particles start at every integer point; we use $\Ph_{\N}$ and $\Ph_{\Z}$ for this case, and define $\psi(p)$ analogously with $\theta(p)$, i.e.\
\[\psi(p):=\Ph_\Z(\stay_0\text{ survives}\mid\stay_0)=(1-\qh)^2.\]
Although the present methods fail to identify the exact threshold, they still suffice to prove the existence of a subcritical phase and furthermore to improve the known upper bounds on the threshold.

The methods of Section~\ref{sec:algebraic} with minor modifications yield the following algebraic identities.
\begin{lemma}\label{lem:discrete-identity}\begin{enumerate}[a)]
\item$\rh=p\qh(1-\qh)$.
\item$\sh=\frac12p\qh^2-\frac12\Ph_\R(D=D'<\infty)$, where $D,D'$ are independent random variables distributed as $D\defeq\min\{x_k:0\from\come_k\}$.
\item$\qh=\frac{1-p}2+p\qh^2+\sh+\Bigl(\frac{1-p}2-\rh-\sh\Bigr)\qh$.
\end{enumerate}
\end{lemma}

\begin{proof}a) proceeds exactly as in Lemma~\ref{flip}. For b), events involving triple collisions at $x_k$ no longer count, and so \eqref{half-half} becomes
\[\Ph\bigl((0\from\bullet)\wedge(\go_{1}\tto\stay)\bigr) = p\Ph\bigl((D<\infty)\wedge(D'<\infty)\wedge(D<D')\bigr),\]
leading to the required result. For c), the proof proceeds as in Lemma~\ref{lem:identity}; \eqref{hits} becomes
\[\frac{1-p}{2}=\Ph(\go_{1}\tto\stay)+\Ph(\go_{1}\collide\come)+\Ph(\go_1\triple),\]
and since 
\[\Ph(0\from\bullet\mid\go_1\collide\come)=\Ph(0\from\bullet\mid\go_1\triple)=\qh,\]
\eqref{use-r} is unchanged except for replacing $q$ by $\qh$, etc.\end{proof}

\begin{proposition}\label{prop:dichotomy_Z}
For all $p\in(0,1)$, either $\qh=1$ or $\qh=-1+\sqrt{\frac1p-\sigma}$,
where
\[\sigma=\Ph(D=D'<\infty).\]
\end{proposition}

\begin{proof}
From Lemma~~\ref{lem:discrete-identity}, it immediately follows that
\begin{align*}
0
	& =p\qh^3+p\qh^2-(1+p(1-\sigma))\qh+(1-p(1+\sigma))\\
	& =-(1-\qh)(p\qh^2+2p\qh+(\sigma+1)p-1),
\end{align*}
hence the assumption $\qh<1$ and the condition $\qh\ge0$ imply that $\qh$ is the positive root of the above second polynomial factor, i.e.\ $\qh=-1+\sqrt{\frac1p-\sigma}$. 
\end{proof}

Clearly $\sigma$ depends on the precise distribution $m$. For the remainder of this section, we consider specifically the case of $m=\delta_1$, which was the only discrete case considered in \cite{junge2}. In this case we have
\[\sigma=\sum_{k\geq 1}\Ph_{\N}(A=k)^2.\]

Let us introduce computable bounds on $\sigma$, so as to deduce effective criteria for survival or extinction of static particles. Define, for all $K\in\N$, 
\[\sigma_K=\sum_{k=1}^K\Ph_{\N}(A=k)^2\qquad\text{and}\qquad \sigmat_K=\sum_{k=1}^K\Ph_{\N}(A=k),\]
which are polynomials in $p$ of degree at most $2K$ and $K$ respectively, and notice that
\[\sigma_K\le\sigma\le \sigma_K+\biggl(\sum_{k=K+1}^\infty\Ph_{\N}(A=k)\biggr)^2=\sigma_K+(\qh-\sigmat_K)^2.\]
Define also
\[\pc[\hat]-\defeq\inf\{p\in [0,1]:\psi(p)>0\},\qquad\text{and}\qquad
\pc[\hat]+\defeq\sup\{p\in [0,1]:\psi(p)=0\}.\]

\begin{theorem}
For any integer $K\ge1$, 
\begin{enumerate}[a)]
	\item\label{survival_Z} If $p(4+\sigma_K)>1$ for all $p>p_0$ for some $p_0\in[0,1]$, then $\psi(p)>0$ for all $p>p_0$;
	\item\label{extinction_Z} If $p(5+\sigma_K-2\sigmat_K+\sigmat_K^2)\le1$, then $\psi(p)=0$.
\end{enumerate}
Therefore, writing $r_K^-$ for the smallest positive root of $p(5+\sigma_K-2\sigmat_K+\sigmat_K^2)=1$ and $r_K^+$ for the largest positive root of $p(4+\sigma_K)=1$, we have $r_K^-\le\pc[\hat]-\le\pc[\hat]+\le r_K^+$.
\end{theorem}

\begin{proof}
\mbox{\ref*{survival_Z})} The condition $p(4+\sigma_K)>1$ implies $p(4+\sigma)>1$, which in turn entails $-1+\sqrt{\frac1p-\sigma}<1$. In such cases, the previous proposition implies a dichotomy similar to that of the continuous case. The same continuity and connectivity arguments as in the continuous case then adapt seamlessly and give the conclusion.

\ref*{extinction_Z}) Since $\sigma\le \sigma_K+\qh-\sigmat_K$, the identity solved by $\qh$ in the proof of Proposition~\ref{prop:dichotomy_Z} yields, if $\qh<1$, 
\[0\le 2p\qh^2+(2-2\sigmat_K)p\qh+(\sigmat_K^2+\sigma_K)p-1,\]
and this polynomial in $\qh$ has a negative root, hence it is positive at all values larger than $\qh$ and in particular at $1$, which means that $(5+\sigma_K-2\sigmat_K+\sigmat_K^2)p-1>0$.
\end{proof}

\begin{corollary}
\begin{enumerate}
[a)]
	\item For all $p\ge1/4$, $\qh<q$, hence $\psi(p)>\theta(p)$;
	\item $0.2354<\pc[\hat]-\le\pc[\hat]+<0.2406$.
\end{enumerate}
\end{corollary}

\begin{proof}
a) It follows from part~\ref{survival_Z}) of the previous theorem and the fact that $\sigma_K>0$ that $\qh<1$. The comparison between $\qh$ and $q$ then follows from the formula for $\qh$ in Proposition~\ref{prop:dichotomy_Z}. 

b) These bounds follow by taking $K=25$ and evaluating and studying numerically the degree~$50$ polynomials $p\mapsto p(4+\sigma_{25})-1$ and $p\mapsto p(5+\sigma_{25}-2\sigmat_{25}+\sigmat_{25}^2)-1$. 
\end{proof}
\begin{remark}The explicit bounds derived above change quite slowly as $K$ is increased. However, the upper bound appears to change much more slowly than the lower, suggesting that if there is a single critical probability, its value is likely to be close to the upper bound given.
\end{remark}

\section*{Acknowledgements}
J.H.\ has been supported by the European Research Council (ERC) under
the European Union's Horizon 2020 research and innovation programme (grant 
agreement no.\ 639046) and by the UK Research and Innovation Future Leaders Fellowship MR/S016325/1, 
and is grateful to Agelos Georgakopoulos for several helpful discussions. 
L.T.\ was supported by the French ANR project MALIN (ANR-16-CE93-000). 

\bibliographystyle{acm}
\bibliography{biblio}

\begin{thebibliography}{10}

\bibitem{arratia1983site}
{\sc Arratia, R.}
\newblock Site recurrence for annihilating random walks on {$\mathbf Z_d$}.
\newblock {\em The Annals of Probability 11}, 3 (1983), 706--713.

\bibitem{belitsky1995ballistic}
{\sc Belitsky, V., and Ferrari, P.~A.}
\newblock Ballistic annihilation and deterministic surface growth.
\newblock {\em Journal of Statistical Physics 80}, 3-4 (1995), 517--543.

\bibitem{ben1993decay}
{\sc Ben-Naim, E., Redner, S., and Leyvraz, F.}
\newblock Decay kinetics of ballistic annihilation.
\newblock {\em Physical Review Letters 70}, 12 (1993), 1890--1893.

\bibitem{junge3}
{\sc Benitez, L., Junge, M., Lyu, H., Redman, M., and Reeves, L.}
\newblock Three-velocity coalescing ballistic annihilation.
\newblock arXiv preprint arXiv:2010.15855, 2020.

\bibitem{bramson1991asymptotic}
{\sc Bramson, M., and Lebowitz, J.~L.}
\newblock Asymptotic behavior of densities for two-particle annihilating random
  walks.
\newblock {\em Journal of Statistical Physics 62}, 1--2 (1991), 297--372.

\bibitem{BM19}
{\sc Broutin, N., and Marckert, J.-F.}
\newblock The combinatorics of the colliding bullets.
\newblock {\em Random Structures and Algorithms 56}, 2 (2020), 401--431.

\bibitem{junge2}
{\sc Burdinski, D., Gupta, S., and Junge, M.}
\newblock The upper threshold in ballistic annihilation.
\newblock {\em ALEA Lat. Am. J. Probab. Math. Stat. 16}, 2 (2019), 1077--1087.

\bibitem{canfield84}
{\sc Canfield, E.~R.}
\newblock Remarks on an asymptotic method in combinatorics.
\newblock {\em Journal of Combinatorial Theory, Series A 37\/} (11 1984),
  348--352.

\bibitem{denisov08}
{\sc Denisov, D., Dieker, A.~B., and Shneer, V.}
\newblock Large deviations for random walks under subexponentiality: the
  big-jump domain.
\newblock {\em The Annals of Probability 36}, 5 (2008), 1946--1991.

\bibitem{droz1995ballistic}
{\sc Droz, M., Rey, P.-A., Frachebourg, L., and Piasecki, J.}
\newblock Ballistic-annihilation kinetics for a multivelocity one-dimensional
  ideal gas.
\newblock {\em Physical Review E 51}, 6 (1995), 5541--5548.

\bibitem{junge}
{\sc Dygert, B., Kinzel, C., Junge, M., Raymond, A., Slivken, E., and Zhu, J.}
\newblock The bullet problem with discrete speeds.
\newblock {\em Electron. Commun. Probab. 24\/} (2019), Paper No. 27, 11.

\bibitem{EF85}
{\sc Elskens, Y., and Frisch, H.~L.}
\newblock Annihilation kinetics in the one-dimensional ideal gas.
\newblock {\em Physical Review A 31}, 6 (1985), 3812--3816.

\bibitem{ermakov1998some}
{\sc Ermakov, A., T{\'o}th, B., and Werner, W.}
\newblock On some annihilating and coalescing systems.
\newblock {\em Journal of Statistical Physics 91}, 5 (1998), 845--870.

\bibitem{flajolet-sedgewick}
{\sc Flajolet, P., and Sedgewick, R.}
\newblock {\em Analytic combinatorics}.
\newblock Cambridge University Press, 2009.

\bibitem{zbMATH00195103}
{\sc {Grimmett}, G.}
\newblock {\em {Percolation}}.
\newblock Springer-Verlag, 1989.

\bibitem{HT}
{\sc Haslegrave, J., and Tournier, L.}
\newblock Combinatorial universality in three-speed ballistic annihilation.
\newblock In {\em In and Out of Equilibrium 3: Celebrating Vladas
  Sidoravicius}. Springer, 2021.
\newblock Progress in Probability 77.

\bibitem{junge-lyu-asymmetric}
{\sc Junge, M., and Lyu, H.}
\newblock The phase structure of asymmetric ballistic annihilation.
\newblock arXiv preprint arXiv:1811.08378, 2018.

\bibitem{ibm}
{\sc Kleber, M., and Wilson, D.}
\newblock ``{P}onder {T}his'' {IBM} research challenge.
\newblock
  \url{https://www.research.ibm.com/haifa/ponderthis/challenges/May2014.html},
  May 2014.

\bibitem{kovchegov2017dynamical}
{\sc Kovchegov, Y., and Zaliapin, I.}
\newblock Dynamical pruning of rooted trees with applications to 1-d ballistic
  annihilation.
\newblock {\em Journal of Statistical Physics 181}, 2 (2020), 618--672.

\bibitem{krapivsky1995ballistic}
{\sc Krapivsky, P.~L., Redner, S., and Leyvraz, F.}
\newblock Ballistic annihilation kinetics: The case of discrete velocity
  distributions.
\newblock {\em Physical Review E 51}, 5 (1995), 3977--3987.

\bibitem{krug1988universality}
{\sc Krug, J., and Spohn, H.}
\newblock Universality classes for deterministic surface growth.
\newblock {\em Physical Review A 38}, 8 (1988), 4271--4283.

\bibitem{nagaev81}
{\sc Nagaev, S.~V.}
\newblock On the asymptotic behaviour of one-sided large deviation
  probabilities.
\newblock {\em Teoriya Veroyatnostei i ee Primeneniya 26}, 2 (1981), 369--372.

\bibitem{sidoravicius-tournier}
{\sc Sidoravicius, V., and Tournier, L.}
\newblock Note on a one-dimensional system of annihilating particles.
\newblock {\em Electronic Communications in Probability 22}, 59 (2017), 9 pp.

\bibitem{siegel}
{\sc Siegel, C.}
\newblock {\em Topics in Complex Function Theory, Vol.\ I: Elliptic Functions
  and Uniformization Theory.}
\newblock Wiley -- Interscience, 1969.

\bibitem{trizac2002kinetics}
{\sc Trizac, E.}
\newblock Kinetics and scaling in ballistic annihilation.
\newblock {\em Physical Review Letters 88}, 16 (2002), 160601.

\end{thebibliography}

\end{document}